\documentclass[11pt]{amsart}

\usepackage{a4wide}
\usepackage{amsmath, amsfonts, amscd, amssymb, amsthm}
\usepackage{tikz}
\usepackage{tikz-cd}
\usepackage[UKenglish]{babel}
\usetikzlibrary{matrix,arrows,calc,decorations.pathmorphing}
\usepackage{xspace}
\usepackage{color}

\newcommand{\thedocumentname}{Resolutions of proper Riemannian Lie Groupoids}
\newcommand{\theauthor}{H.\ Posthuma, X.\ Tang, and K.J.L.\ Wang}

\usepackage[pdftitle={\thedocumentname},pdfauthor={\theauthor},bookmarks,urlcolor=black,colorlinks=true,linkcolor=black,citecolor=black]{hyperref}
\usepackage{aliascnt}

\addto\extrasUKenglish{%

}

\numberwithin{equation}{section}							
\theoremstyle{plain}
\newtheorem{thm}{Theorem}[section]

\newaliascnt{lem}{thm}
\newtheorem{lem}[lem]{Lemma}
\aliascntresetthe{lem}

\newaliascnt{prop}{thm}
\newtheorem{prop}[prop]{Proposition}
\aliascntresetthe{prop}               

\newaliascnt{cor}{thm}
\newtheorem{cor}[cor]{Corollary}
\aliascntresetthe{cor}

\theoremstyle{definition}
\newaliascnt{defn}{thm}
\newtheorem{defn}[defn]{Definition}
\aliascntresetthe{defn}

\newaliascnt{rem}{thm}
\newtheorem{rem}[rem]{Remark}
\aliascntresetthe{rem}

\newaliascnt{exa}{thm}
\newtheorem{exa}[exa]{Example}
\aliascntresetthe{exa}

\newtheorem*{thm*}{Theorem}

\newcommand{\rar}{\rightarrow}
\newcommand{\drar}{\dashrightarrow}
\newcommand{\lar}{\leftarrow}
\newcommand{\rrar}{\rightrightarrows}
\newcommand{\bpm}{\begin{pmatrix}}
\newcommand{\epm}{\end{pmatrix}}

\newcommand{\G}{\mathcal{G}}
\newcommand{\h}{\mathcal{H}}

\newcommand{\codim}{\mbox{codim}}
\newcommand{\wt}{\widetilde}

\begin{document}
\author{H.\ Posthuma}
\address{Korteweg-de Vries Institute for Mathematics, University of Amsterdam, The Netherlands}
\email{H.B.Posthuma@uva.nl}
\author{X.\ Tang}
\address{Department of Mathematics, Washington University, St. Louis, USA}
\email{XTang@math.wustl.edu}
\author{K.J.L.\ Wang}
\address{Korteweg-de Vries Institute for Mathematics, University of Amsterdam, The Netherlands}
\email{K.J.L.Wang@uva.nl}
%
%
\begin{abstract} 
In this paper we prove that every proper Lie groupoid admits a desingularization to a regular proper Lie groupoid. When equipped with a Riemannian metric, we show that it admits a desingularization to a regular Riemannian proper Lie groupoid, arbitrarily close to the original one in the Gromov--Hausdorff distance between the quotient spaces. We construct the desingularization via a successive blow-up construction on a proper Lie groupoid.  We also prove that our construction of the desingularization is invariant under Morita equivalence of groupoids, showing that it is a desingularization of the underlying differentiable stack.
\end{abstract}
\title{Resolutions of proper Riemannian Lie groupoids}
\maketitle
\vspace{-2em}
\tableofcontents
\vspace{-3.5em}
\section{Introduction}
\label{sec:introduction}

A groupoid is a small category with invertible morphisms. Over the last 50 years, groupoids, a generalization of groups, have become a fundamental tool in the study of spaces with singularities. 

Let $\mathcal{G}$ be the set of morphisms of the category, and $M$ be the set of objects. The set $M$ is equipped with a canonical equivalence relation. Two objects $x$ and $y$ in $M$ are equivalent if and only if there is a morphism $g\in \mathcal{G}$, such that the source $s(g)$ of $g$ is $x$ and the target $t(g)$ of $g$ is $y$.  Accordingly, the set $M$ is naturally decomposed into a disjoint union of subsets of objects in the same equivalence class. Each such subset is called a leaf, or an orbit, of the groupoid $\mathcal{G}$. The quotient space of $M$ with respect to this equivalence relation is called the orbit space, or the leaf space, of $\mathcal{G}$. Many interesting examples of singular spaces in geometry can be described by orbit spaces of Lie groupoids, e.g.\ orbifolds, leaf spaces of foliations, moduli spaces of curves, etc.

A groupoid is called proper when the structure map $\mathcal{G}\to M\times M,\ g\mapsto (s(g), t(g))$, is proper. Proper Lie groupoids can be viewed as the groupoid generalizations of proper Lie group actions. Proper Lie groupoids share many nice properties with proper Lie group actions. For example, a proper Lie groupoid  can be linearized near an orbit, its object space is stratified accroding to orbit type \cite{Weinstein02,Zung06, PflaumPosthumaTang14,CrainicStruchiner13,HoyoFernandes14}, etc.

If the dimension of the leaves of a groupoid $\mathcal{G}$ is constant, the groupoid is said to be regular. {Moerdijk} obtained in \cite{Moerdijk02} a beautiful structure theorem about regular Lie groupoids by showing that they fit into an exact sequence   
\[
K\rar \G\rar E,
\] 
where \(K\) is a bundle of Lie groups and \(E\) is a foliation groupoid. In case \(\G\) is proper, the groupoid \(E\) is an orbifold groupoid, and the fibers of $K$ are compact Lie groups. In general, a proper Lie groupoid may have leaves of varying dimension, and the corresponding orbit space may have more complicated singularity than that of an orbifold. In this paper, we aim to answer the  following natural question:

\medskip

\begin{center}
{\em ``How close is a proper Lie groupoid to a regular one?"}
\end{center}

\medskip

It is well-known, c.f.\ \cite{DuitstermaatKolk00,am}, that a proper action of a Lie group $G$ on a manifold $M$ admits a ``desingularization"  $\widetilde{M}$ satisfying the following two properties:
\begin{enumerate}
\item  $\widetilde{M}$  is a smooth manifold equipped with a proper $G$-action whose orbits all have the same dimension;
\item  There is a $G$-equivariant surjective map $p\colon \widetilde{M}\to M$ which is a diffeomorphism on an open dense subset of $M$. 
\end{enumerate}

In the first main theorem of this paper (\autoref{thm:action-resolution} and  \autoref{cor:morita-blow-up}), we present a generalization of this construction to a general proper Lie groupoid. 

\begin{thm}[\autoref{thm:action-resolution} and \autoref{cor:morita-blow-up}]
\label{thm:main1}
Any proper Lie groupoid  admits a desingularization to a regular proper Lie groupoid. Moreover, Morita equivalent proper Lie groupoids admit Morita equivalent desingularizations.\end{thm}

The main tool we use in the proof of the above  \autoref{thm:main1} is the blow-up construction for Lie groupoids. The blow-up construction for Lie groupoids has been studied extensively in the literature, e.g.\ \cite{GualtieriLi14, Nistor15, DebordSkandalis17}.
In this paper, we restrict our attention to full subgroupoids of a proper Lie groupoid. The advantage of proper Lie groupoids is that they can be linearized (c.f.\ \cite{HoyoFernandes14}) around closed saturated submanifolds \(S\subset M\). More concretely, a proper Lie groupoid is locally isomorphic to a linear action groupoid of a proper Lie groupoid. As the blow-up is a local construction, we are therefore able to give a relatively simple expression for the blow-up of the groupoid than the general case, which naturally generalizes the one of Duistermaat and Kolk in \cite{DuitstermaatKolk00}.

\autoref{thm:main1} shows that at the differential topological level, a proper Lie groupoid is ``almost" a regular proper Lie groupoid. To improve our understanding of such a desingularization process, in the second half of this paper we investigate Riemannian geometry on proper Lie groupoids and their desingularizations. 
Our study is inspired by Alexandrino's results on desingularizations of singular Riemannian foliations. In \cite{Alexandrino10}, by a successive blow-up construction, Alexandrino constructed a regular Riemannian foliation that desingularizes the original singular Riemannian foliation. In \cite{PflaumPosthumaTang14}, Pflaum and the first two authors observed that the unit space of every proper Lie groupoid carries a singular Riemannian foliation. In \cite{HoyoFernandes14}, del Hoyo and Fernandes introduced the concept of Riemannian groupoid and proved that every proper Lie groupoid is Riemannian. With these developments, it is natural for us to improve Alexandrino's theorem, \cite{Alexandrino10}, in the case of proper Lie groupoids in the framework of Riemannian groupoids. 

\begin{thm}[\autoref{thm:metric-desingularization} and \autoref{thm:riemannian-morita-desingularization}] 
\label{thm:main2}
Any Riemannian proper groupoid admits a Riemannian desingularization. Moreover, Morita equivalent Riemannian groupoids admit Morita equivalent desingularizations.
\end{thm}
It should be remarked that in general it is not difficult to define a metric on the desingularization of a proper Lie groupoid, however:
\begin{itemize}
\item this metric will not be simplicial, i.e., compatible with the groupoid structure,
\item the two quotient spaces, viewed from the point of view of metric topology, may be quite far away from each other.
\end{itemize}
The point of our Theorem above is that our proof takes care of both these issues. It turns out that in the context of proper Lie groupoids, the original proof of Alexandrino can be significantly simplified in our case. Furthermore, the Riemannian groupoid structure in  \autoref{thm:main2} naturally strengthens the property of being a singular Riemannian foliation associated to a proper Lie groupoid.  

The Morita invariant properties in  \autoref{thm:main1} and  \autoref{thm:main2} show that the (Riemannian) desingularization of a proper (Riemannian) groupoid is an invariant of  the underlying (Riemannian) differentiable stack associated to a proper Lie groupoid. In the future, we plan to use the desingularization introduced in this paper to investigate topological and geometric properties of the underlying stack. 

\subsection*{Organization of the paper}

We start with recalling some basics on groupoids in \autoref{sec:background}.  Then we continue in \autoref{sec:stratifications} by showing that \(M\), the space of objects, admits a stratification, where each stratum has leaves of a fixed dimension. Before we use these stratifications, we first briefly discuss general resolutions in \autoref{sec:resolutions}, of which the desingularizations are special examples. Then we will define the blow-up construction in \autoref{sec:blow-up} and use the stratification to prove  \autoref{thm:main1}. In \autoref{sec:metric} we recall the definition of a simplicial metric on a groupoid and continue with proving  \autoref{thm:main2}. We end the paper with a discussion in \autoref{sec:discussion}.

\subsection*{Acknowledgements}
The authors would like to thank Marcos Alexandrino, Marius Crainic, Rui Fernandes, Matias del Hoyo, Markus Pflaum, David Martinez Torres for helpful discussion. 
In a late stage of this work we were informed that David Martinez Torres has considered resolutions similar to those in \S \ref{sec:blow-up} 
in the context of Poisson manifolds of compact type. Tang's research is partially supported by NSF DMS 1362350. The research of Wang is supported by NWO TOP nr. 613.001.302.

%
\section{Background}
\label{sec:background}

We start by introducing some notation. Given a groupoid \(\G\
\rrar M\), we denote its source and target maps by \(s\) and \(t\) respectively and the unit map \(M\rar\G\) by \(u\). A groupoid is called \emph{proper} if the combined target and source map \((t,s):\G\rar M\times M\) is a proper map. If this is the case, then \(\G\rar M\) admits proper Haar systems. 

Recall that a \emph{Haar system} on \(\G\) is  a family \(\{\mu^x\}_{x\in M}\) of measures on \(s^{-1}(x)\) which are \emph{right-invariant}: \(R_g^*(\mu^{s(g)})=\mu^{t(g)}\) for \(g\in \G\), and \emph{smooth}: for all \(f\in C_c^\infty(\G)\), the map \(x\mapsto \mu^x(f|_{s^{-1}(x)})\) is smooth. Such a Haar system \(\mu\) is called \emph{proper} if the source map restricted to the support of \(\mu\) is a proper map. Proper Haar systems will later on be used to perform an averaging procedure. 

Given any map \(f\colon N\rar M\), we say that \(\G\) \emph{acts on} \(N\), or \(N\) admits a  \emph{\(\G\)-action}, whenever there is a smooth map \(\theta\colon\G\times_M N\rar N\) , \(\theta(g,n) =  g\cdot n\) such that \(f(g\cdot n) = t(g)\). Given a \emph{connection} \(\sigma\) on \(\G\), i.e., a vector bundle morphism  \(\sigma\colon s^*TM\rar T\G\) such that \(ds\circ \sigma = id_{s^*TM}\) and \(\sigma|_M = du\), we can lift this action to a quasi-action of \(\G\) on \(TN\), defined as:
\begin{align}
	\label{eg:lift-action-tangent}
	&
	g\cdot w:= d\theta(\sigma_g \circ df(w),w ),
	& \forall w\in TN.
\end{align}
We will call this the \emph{tangent lift} of \(\theta\), and denote it by \(T\theta\).

The morphisms between groupoids we will use are the so-called right-principal bibundles. Recall that a \emph{bibundle} between two groupoids \(\G'\) and \(\G\) consists of a manifold \(P\) with smooth maps \(\alpha\colon P\rar M\) and \(\alpha\colon P'\rar M'\), called \emph{moment maps} on which the groupoids act.
\begin{center}
	\begin{tikzcd}
		\G' 
		\arrow[shift left,d] \arrow[d,shift right] \arrow[r,bend left=45]
		&
		P 
		\arrow{dr}[swap]{\alpha} \arrow{dl}{\alpha'}
		&
		\G 
		\arrow[d,shift left] \arrow[d,shift right] \arrow[l,bend right=45]
		\\
		M'
		& 
		&
		M
	\end{tikzcd}
\end{center}
These two actions commute and each moment map is invariant with respect to the other action. If the map \(\alpha'\) is a submersion, the action of \(\G\) on \(\alpha\) is free and proper, and its orbits are the fibers of \(\alpha'\), we call the bibundle \emph{right-principal}. Left-principal bibundles are defined similarly.  Two bibundles \(P\) and \(Q\) are \emph{diffeomorphic} as bibundles if they are diffeomorphic as manifolds and the diffeomorphism commutes with all the bibundle maps and actions.

\begin{defn}
	A \emph{generalized morphism} \([P]\colon\G'\drar\G\) is an isomorphism class of right principal bibundles between \(\G'\) and \(\G\).
\end{defn}

The easiest example of a generalized morphism is given by an actual morphism of groupoids: if  \(\G'\rar \G\) is a  groupoid morphism, then \(P:=M'\times_{M} \G\) defines a generalized morphism \([P]\colon\G'\dashrightarrow \G\).
From now on, we will denote the isomorphism class of \(P\) by the same symbol \(P\). Given two generalized morphisms \(Q\colon\G'' \drar \G'\) and \(P\colon\G'\drar \G\), a representative of their composition is given by \((Q\times_{M'} P)/\G'\colon\G''\drar\G\), with the \(\G'\)-action on \(Q\times_{M'} P \) given by \((q,p)\cdot g' = (q\cdot g', (g')^{-1}\cdot p)\). An isomorphism in the category of groupoids with generalized morphisms as morphisms leads to the notion of Morita equivalence:

\begin{defn}
	Two groupoids are called \emph{Morita equivalent} if there exists a generalized morphism \(P\colon\G'\drar\G\) which is also left-principal.
\end{defn}

The local study of groupoids has a long history, see for example \cite{CrainicStruchiner13,HoyoFernandes14,PflaumPosthumaTang14, Weinstein02, Zung06}. The main results of this local study are linearization results, which we state as  \autoref{thm:lin} and \autoref{cor:lin-sat}. Proofs can for example be found in \cite{HoyoFernandes14}, which uses  the notion of a Riemannian metric on a groupoid, to which we come back in \autoref{sec:metric}.
Recall that for a groupoid \(\G\rrar M\), \(S\subset M\) and \(NS\) its normal bundle is a groupoid
\begin{align}
N_S(\G) = \G_S \times_S NS\rrar NS,
\end{align}
where the target map is given by the action as \(t(g,[(d_gs)(v)])=[(d_gt)(v)]\) for \(v\in T_g\G\). Then the linearization results are:

\begin{thm}
	\label{thm:lin}
	Let \(\G\rrar M\) be a proper Lie groupoid and let \(S\subset M\) be saturated. Then \(\G\) is linearizable around \(S\), i.e., there exist open sets  \(S\subset U \subset M\) and \(S\subset V\subset NS\) such that \(\G|_U \simeq N_S(\G)|_V\).
\end{thm}

\begin{cor}
	\label{cor:lin-sat}
	Let \(\G\rrar M\) be a \(s\)-proper Lie groupoid and let \(S\subset M\) be saturated. Then \(\G\) is invariantly linearizable around \(S\), i.e., there exist saturated open sets \(S\subset U \subset M\) and \(S\subset V\subset NS\) such that \(\G|_U \simeq N_S(\G)|_V\).
\end{cor}

Besides this linearization around a saturated submanifold, one can also linearize a proper Lie groupoid more locally near a point. By \cite{PflaumPosthumaTang14} for all \(x\in M\) there exists open subsets \(U\subset M\), \(O\subset L_x\) and \(V\subset N_xL_x\), with \(L_x\) the leaf through \(x\), such that: \(U\simeq V\times O\) and 
\begin{align}
	\label{eq:local-lin}
	\G|_U
	\simeq 
	(\G_x\times V) \times (O\times O)\rrar V \times O.
\end{align}

%
\section{Stratifications of proper Lie groupoids}
\label{sec:stratifications}
Besides the local picture of the groupoid, the orbit space \(M/\G\) also has a local structure, namely that of a stratified space. A thorough account of stratifications can, for example, be found in \cite{Pflaum01}. We recall:

\begin{defn}
	A \emph{stratification} of a Hausdorff, second-countable paracompact space \(X\) is a locally finite partition \(\{X_i\}\) of \(X\) into locally closed, connected subsets \(X_i\subset X\), called \emph{strata}, such that:
	\begin{itemize}
		\item Each \(X_i\) is a smooth manifold with the induced topology from \(X\);
		\item The closure of each \(X_i\) is a union of \(X_i\) with  strata of lower dimension.
	\end{itemize}
	The second condition is called the \emph{Frontier} condition. If a partition \(\{X_i\}\) of \(X\) satisfies all conditions of a stratification except that the subsets \(X_i\) are not connected, we call \(\{X_i\}\) a \emph{decomposition} of \(X\).
\end{defn}

\begin{lem}
	\label{lem:sep-strata}
	Let \(\{X_i\}\) be a stratification of \(X\) and suppose that \(X_1\) and \(X_2\) have the same dimension. Then there exist open neighbourhoods \(U_i\) of \(X_i\) in \(X\) such that \(U_1\cap U_2=\emptyset\).
\end{lem}
\begin{proof}
	Since the Frontier condition holds, we know that \(\overline{X_1}\cap \overline{X_2}=\emptyset\). Since \(X\) is Hausdorff and paracompact, we can separate closed sets.
\end{proof}

When \(X=M\), a manifold, one usually asks the strata \(X_i\) to be smooth submanifolds. In this case the following lemma holds,  see e.g.  \cite{Mestre16}.

\begin{lem}
\label{lem:max-stratum}
	Let \(\{X_i\} \) be a stratification of a manifold \(M\) and suppose there are no strata of codimension one. Then, there exists a unique maximal stratum, which is dense, open, and connected.
\end{lem}

We will continue by discussing two natural partitions of \(M\), which lead to a saturated stratification of \(M\).

\subsection{Dimensional type}
Let \(\G\rrar M\) be a proper Lie groupoid. Since we want to construct regular groupoids, i.e.\ groupoids for which each leaf has the same dimension, it is natural to consider the following partition of \(M\). For \(0\leq j\leq \dim(M)\) define:
	\begin{align}
	\label{eq:def:strat-dim}
		S^j:=\{x\in M\,|\, \codim(L_x) = j\}.
	\end{align}

The different connected components of the \(S^j\) can have different dimensions. In fact, the dimensions are easily computable.

\begin{lem}
\label{lem:dim-strata}
	Let \(L_x\) be the leaf through \(x\in M\), \(j:=\codim(L_x)\) and let \(x\in S_x\subset S^j\) be its connected component. Then we have:
	\begin{align*}
	\dim(S_x) = \dim(L_x) + \dim\left((N_xL_x)^{\G_x^\circ}\right),
	\end{align*}
	with \(V^G=\{v\in V\,|\, g\cdot v = v,\, \forall g\in G\}\) for a linear action \(G\curvearrowright V\) and \(G^\circ\) is the connected component of \(G\) at the identity.
\end{lem}
\begin{proof}
	Using the local linearization around \(x\in M\), we have that \(\G|_{U_x}\simeq (\G_x\times N_xL_x)\times (O_x\times O_x)\rrar N_xL_x \times O_x\). Hence for any \(v\in N_xL_x\), we have that its leaf locally looks like \(\G_x\cdot v\times O_x\). Therefore:
	\begin{align*}
	\dim(L_v) = \dim(\G_x\cdot v) + \dim(O_x).
	\end{align*}
	This is equal to \(\dim(O_x)\) if and only if \(\G_x\cdot v\) is discrete, i.e.\ \(v\in (N_xL_x)^{\G_x^\circ}\). Hence \(S_x\cap U \simeq (N_xL_x)^{\G_x^\circ}\times O_x\), which proves the lemma.
\end{proof}

\begin{prop}
\label{prop:strat-dim}
	The connected components of \(S^j\) form a stratification of \(M\).
\end{prop}
\begin{proof}
	Let \(S\subset S^j\) be a connected component and let \(y\in \bar{S}\setminus S\). Local linearization around \(y\) gives an open neighbourhood \(U_y\) of \(y\) such that
	\begin{align*}
	\G|_{U_y}\simeq (\G_y \times N_y L_y) \times (O_y\times O_y)\rrar N_yL_y\times O_y\simeq U_y.
	\end{align*}
	Since \(y\in \bar{S}\) there exists an \(x\in S\cap U_y\). Without loss of generality, we can assume that \(x=v\in N_yL_y\). We immediately realize that \(L_x\cap U_y\simeq \G_y\cdot v \times O_y\). Since \(S\) is connected, we have that \(\dim(L_x)\neq \dim(L_y)\) and therefore \(\dim(L_x) > \dim(L_y)\) and \(G_y\cdot v\) is not discrete. Hence \(G_y\cdot v \cap (N_yL_y)^{\G_y^\circ}=\emptyset\) and we can view \((N_yL_y)^{\G_y^\circ}\) as a subspace of \(N_xL_x = N_v(N_yL_y)\). Moreover, \(\G_x\simeq \mbox{Stab}(v)<\G_y\), the stabilizer of \(v\) with respect to the \(\G_y\)-action.  Hence \((N_yL_y)^{\G_y^\circ}\) is in fact a subspace of \((N_xL_x)^{\G_x^\circ}\). Therefore, we find that:
	\begin{align*}
	\dim(S_x)
	& 
	= \dim(L_x) + \dim((N_xL_x)^{\G_x^\circ}) 
	> \dim(L_y) + \dim((N_yL_y)^{\G_y^\circ})
	= \dim(S_y).
	\end{align*}
	To conclude that the Frontier condition holds, we need to show that \(S_y\), the stratum through \(y\), lies completely inside \(\bar{S}\). Let \(A=S_y\cap \bar{S}\), so that \(A\subset S_y\). Since \(\bar{S}\) is closed, \(A\) is a closed subset of \(S_y\). We will show that it is open as well and hence by the connectedness of \(S_y\), we conclude that \(A=S_y\). In the notation as above, if \(y\in \bar{S}\) and \(v\in N_yL_y\) belongs to \(S\), then for all \(y'\in O_y\), all \(w\in (N_yL_y)^{\G_y^\circ}\) and all \(\lambda>0\), we have that \((w+ \lambda v, y')\in S\). Hence \((w,y')\in \bar{S}\). Note that \(S_y\cap U_y = (N_yL_y)^{\G_y^\circ}\times O_y\). Hence \(\bar{S}\) is indeed open in \(S_y\) and therefore the Frontier condition holds. Note that for all \(y\), we have that \(S_y\cap U_y  = (N_yL_y)^{\G_y^\circ}\times O_y\) also shows that the \(S_y\) are embedded submanifolds of \(M\). This concludes the proposition.
\end{proof}

\begin{defn}
	Let \(\G\rrar M\) be a proper Lie groupoid. The stratification by connected components of the subsets \(S^j\) is called the \emph{dimension stratification} of \(\G\).
\end{defn}

\begin{prop}
\label{prop:properties-dim-strat}
	Let \(\G\rrar M\) be a proper Lie groupoid and define the integers  \(j=\max\{0\leq i \leq \dim(M)\,|\,S^i\neq\emptyset \}\) and \(m =\min\{0\leq i \leq \dim(M)\,|\,S^i\neq\emptyset \}\). Then the following properties hold:
	\begin{enumerate}
		\item Each stratum \(S\subset S^j\) is closed;
		\item For all \(S_1,S_2\subset S^j\) there exist open neighbourhoods \(S_i\subset U_i\subset M\) such that \(U_1\cap U_2=\emptyset\);
		\item Each stratum \(S\subset (S^m)^c\)  has \(\dim(S)  <  \dim(M)-1\);
		\item \(S^m\) is open, dense and connected;
		\item \(x\in S^m\) if and only if \(\G^\circ_x\) acts trivially on \(N_xL_x\).
	\end{enumerate}
\end{prop}
\begin{proof}
	Using the proof of \autoref{prop:strat-dim}, we know that the leaf of any \(y\in \bar{S}\setminus S\) has higher co-dimension than the leaves in \(S\) itself. Since \(j\) is chosen as the maximum, we conclude property (1). Property (2) then follows immediately as we can separate closed sets.
	For property(3), note that \autoref{lem:dim-strata} implies that if \(\dim(S_x) =\dim(M)-1 \), \(N_xL_x/ (N_xL_X)^{\G_x^\circ}\) is one-dimensional and therefore generated by a vector \(v\in N_xL_x\). Therefore, there exists \(g\in \G_x^\circ\) such that \(g\cdot v \neq v \) and \(g\) acts trivially on the rest of \(N_xL_x\). As the linearization of \(\G\) is through a metric, we know that \(||g\cdot v|| = ||v||\) and hence \(g\cdot v = -v\). But this shows that \(\G_x^\circ\) has a \(\mathbb{Z}_2\) component, which is a contradiction with \(\G_x^\circ\) being connected. Hence there are no codimension-one strata. Now \autoref{lem:max-stratum} implies property (4). The final property is now easily realized by the denseness of \(S^m\).
\end{proof}

\subsection{Reduced normal orbit type}
It is obvious that the action of \(\G_x^\circ\) plays a natural role in stratifying \(M\) by dimension. 
Again, let \(\G\rrar M\) be a proper Lie groupoid. In this section, we will show that the dimension stratification comes naturally from a decomposition, which is defined using the \(\G_x^\circ\)-action.  As the results of this section give exactly the same stratification as the dimension stratification, it is not necessary to read this section in order to understand the rest of the paper.

Given a connected Lie group \(G\) acting linearly on a vector space \(V\), we define:
	\begin{align}
		S^{(G,V)}:=\{x\in M\,|\,( G_x^\circ \curvearrowright N_xL_x ) \simeq (G\curvearrowright V)\}.
	\end{align}
Note that these subsets are smooth saturated manifolds. Moreover, if \(x\in S^{(G,V)}\) for some pair \((G,V)\), we immediately get that 
	\begin{align*}
		\codim(L_x) 
		= 
		\dim(N_xL_x)
		= 
		\dim (V),
	\end{align*}
and therefore
	\begin{align}
		S^{(G,V)}\subset S^{\dim(V)}.
	\end{align}

Moreover,  \autoref{lem:dim-strata} shows that all the connected components of the \(S^{(G,V)}\) have the same dimension.

\begin{prop}
	The set of \(S^{(G,V)}\) form a decomposition of \(M\). Moreover, it connected components agree with the connected components of the \(S^j\) defined in 	\autoref{eq:def:strat-dim}. Hence the corresponding stratification is the same as the dimension stratification.
\end{prop}
\begin{proof}
	Using the local linearization, similarly to the proof of  \autoref{prop:strat-dim}, one immediately realizes that if \(\bar{S}\cap S'\neq\emptyset\), then \(S'\subset \bar{S}\). Once again, using  \autoref{lem:dim-strata}, we conclude that the Frontier condition holds. Hence the partition into sets \(S^{(G,V)}\) forms a decomposition. Now, let \(S\) be a connected component of \(S^{(G,V)}\) and let \(k=\dim(V)\). For any \(y\in\bar{S}\subset S^k\),  an open neighbourhood \(U_y\subset M\) of \(y\) such that \(\G|_U\simeq (\G_y\times N_y L_y ) \times( O_y\times O_y)\), and \(v\in N_yL_y\) corresponding to \(x\in S\cap U_y\), we have that \(\G_x \simeq \mbox{Stab}(v)< \G_y\). This is a closed Lie subgroup, which is of the same dimension. Hence \(\G_x^\circ = \G_y^\circ\) and \(y\in S\), which is therefore closed in \(S^k\). This argument shows that \(S\subset S^k\) is also open: given a leaf in \(S\), any leaf has either a different dimension or the same reduced normal orbit type. That is, for \(y \in S\) we have that \(S\cap U_y = (N_yL_y)^{\G_y^\circ}\times O_y = S^k\cap U_y\). Thus \(S\subset S^k\) is both open and closed and hence connected.
\end{proof}

\begin{defn}
	Let \(\G\rrar M\) be a proper Lie groupoid. The decomposition of \(M\) into the sets \(S^{(G,V)}\) is called the \emph{reduced normal obrbit decomposition}.
\end{defn}

\begin{rem}
	The above decomposition is called reduced, as we only use the connected components of the isotropy groups. One could also use the whole groups, but the induced stratification will be different. An easy example of this phenomenon is the \(S^1\)-action on the (open) Mobi\"us band. Indeed, this groupoid is regular and connected and hence only has one stratum in the dimension stratification. The middle leaf, however, has a \(\mathbb{Z}_2\)-action as isotropy, and all the other leaves have trivial isotropy.
\end{rem}

%
\section{Resolutions of Lie groupoids}
\label{sec:resolutions}
\begin{defn}
	Let \(\G\rrar M\) be a connected Lie groupoid of dimension \(n\). We define the \emph{category of resolutions over \(\G\)}, denoted by \({\rm Res}_\G\), as the category which has as objects \((\h,\pi)\) with \(\h\) a regular Lie groupoid of dimension \(n\) and \(\pi\colon\h\rar \G\) a surjective proper map, which is an isomorphism almost everywhere. An arrow \((\h',\pi')\rar (\h,\pi)\) in the category \({\rm Res}_\G\)  is a generalized morphism \(\h'\rar\h\) which commutes with \(\pi\) and \(\pi'\):
	\begin{center}
		\begin{tikzcd}
			\h' \arrow{dr}[swap]{\pi'}  \arrow[dashed, rr]
			&
			&
			\h \arrow{dl}{\pi}
			\\
			&
			\G
		\end{tikzcd}
	\end{center} 
\end{defn}

Properness of the map in a resolution comes into play in the next lemma:

\begin{lem}
	Let \(\h\rrar N\) be a resolution of \(\G\rrar M\). Then \(\G\) is proper if and only if \(\h\) is.
\end{lem}
\begin{proof}
	Consider the commutative diagram:
	\begin{center}
		\begin{tikzcd}
			\h \arrow{r}{\pi} \arrow{d}[swap]{(s,t)}
			&
			\G \arrow{d}{(s,t)}
			\\
			N\times N \arrow{r}[swap]{\pi_0\times\pi_0}
			&
			M\times M
		\end{tikzcd}
	\end{center}
	Note that the horizontal morphisms are by assumption proper. Using surjectivity, it follows that for \(K_M\subset M\times M\) and \(K_N\subset N\times N\) we have:
	\begin{align*}
	(s,t)^{-1}(K_M) = \pi_1\left(  ( \pi_0\times\pi_0\circ (s,t)) ^{-1}(K_M)\right),
	\\
	(s,t)^{-1}(K_N) \subset ((s,t)\circ\pi_1)^{-1}\left( \pi_0\times\pi_0(K_N)  \right).
	\end{align*}
	The first equation immediately shows that if \(\h\) is proper, so is \(\G\). The second one does the reverse when we consider that closed subsets of compact sets are compact.
\end{proof}

With this definition of a resolution, we allow the existence of several resolutions for a single groupoid. For example, if \(\h\rar \G\) is a resolution, then so is \(\h\times K\) for any finite  group \(K\).  Therefore, we want to be able to measure how large the resolution is.



\begin{defn}
	Let \((\h,\pi)\) be a resolution of \(\G\). We call it an \emph{action resolution} if the map \(\h\rar \G\times_M N\), given by \(h\mapsto (\pi(h),s(h))\), is an isomorphism.
\end{defn}

Action resolutions are the same as regular groupoid actions in the following sense: if \(\h\) is an action resolution, then \(\G\) acts regularly on the proper map \(N\rar M\), and if \(\G\) acts regularly on any proper map \(N\rar M\), then \(\h:= \G\times_M N\) is an action resolution.

\begin{rem}
	Note that \(\h\) being an action resolution does not imply that the map \(\pi\) is submersive. Indeed, as \(\h\) is regular, the isotropy groups are all of the same dimension, but the isotropy groups of \(\G\) change dimension. When we start constructing resolutions, we will see an example of this.
\end{rem}

Before showing existence of (action) resolutions, which is done in the next section, we consider the behaviour of resolutions under Morita equivalences. The following proposition shows that action resolutions are stable under Morita equivalence.

\begin{prop}
	Let \(\G'\) be Morita equivalent to \(\G\) via a principal bibundle \(P\) and let \(\h\) be a action resolution of \(\G\). Then \(\G'\) admits an action resolution \(\h'\) which is Morita equivalent to \(\h\). 
\end{prop}
\begin{proof}
	Let \(Q:=P\times_M N\). Since \(P\) is a Morita equivalence bibundle, we obtain that \(\G'\) is isomorphic to \(P\times_M P/\G\) and \(M'\) to \(P/\G\). Now let \(\h':=Q\times_N Q/\h\) and  \(N':=Q/\h\). As \(\h\) acts freely and properly on \(Q\) by the action \((p,n)\cdot h = (p\cdot\pi(h), s(h))\), the above defined \(\h'\) and \(N'\) are smooth manifolds. Also, it is clear that \(\h'\rrar N'\) has the structure of a Lie groupoid and by construction \(Q\) is a Morita equivalence bibundle between \(\h'\) and \(\h\).

	Define \(\pi'\colon\h'=(Q\times_N Q)/\h\rar\G'\simeq (P\times_M P)/\G\) as the map induced by \(Q\rar P\). It is easy to check that \(\pi'\) is well-defined and surjective. Moreover, since \(N\rar M\) is a proper map, so is \(Q\rar P\). This immediately implies that \(\pi'\) is proper. We are left to show that \(\h'\) is a regular groupoid, which implies that it is a resolution, and to show that it is an action resolution.
	
	The morphism 
	\begin{align*}
	&
	\h'\rar \G'\times_{M'} N',
	&
	[p_2,p_1,n]_\h\mapsto ([p_2,p_1]_\G, [p_1,n]_\h) .
	\end{align*}
	has well-defined inverse defined by
	\begin{align*}
	&
	\G'\times_{M'} N' \rar \h',
	&
	([p_2,p_1]_\G,[p_0,n]_\h)\mapsto [p_2\cdot g,p_0,n]_\h,
	\end{align*}
	with \(g\) such that \(p_1\cdot g = p_0\).  Hence if we are able to show that \(\h'\) is regular, so that it is a resolution, we can immediately conclude that it is an action resolution.
	
	To show that \(\h'\) is regular, we will show that it has isomorphic isotropy groups to \(\h\). A groupoid is regular if and only if the dimension of its isotropy groups is constant, so \(\h\) being regular then implies that \(\h'\) is regular. Let \(n_0'=[q_0]\in N'\) be given for \(q=(p_0,n_0)\in Q\) and let \(\xi\colon\h_{n_0}\rar \h'_{n_0'}\) be defined as \(h\mapsto [q_0\cdot h,q_0]\). It is an injective group morphism. Hence we are left to show that it is surjective.
	Let \([q_2,q_1]\in \h'_{n_0}\) be given. That is, \([q_2]=[q_1]=[q_0]\) and hence there exists \(h_i\in \h\) such that \(q_i\cdot h_i = q_0\). Note that this implies that \(h_i\in \h_{n_0}\) and \([q_2,q_1] = [q_2\cdot h_1, q_0] =\xi(h_1^{-1}h_2)\). So indeed, the proposition holds.
	 
\end{proof}

%
\section{Blow-up and desingularization}
\label{sec:blow-up}
In this section we will define the blow-up of a proper Lie groupoid \(\G\rrar M\) along a saturated submanifold \(S\subset M\) in an explicit manner. When we choose \(S\) to be a stratum of the stratification by dimension of leaves, the resulting groupoid will be `more regular'. In particular, after blowing-up a finite amount of times, we will end up with a resolution of our original groupoid. This will then show that resolutions always exist. Our blow-up construction agrees with the one by Debord and Skandalis in \cite{DebordSkandalis17}. Let us start by recalling the real-projective blow-up of a manifold.
\begin{defn}
	Let \(M\) be a manifold, \(S\subset M\) a closed submanifold and \(\phi\colon V\subset NS \rar U \subset M\) a tubular neighbourhood of \(S\). Then, the \emph{blow-up \(\wt{M}\) of \(M\) along \(S\)} is given by the manifold \(\wt(M\), where:
	\begin{align}
	&
	\wt{M} := M\setminus S \cup_\phi \wt{V};
	\\
	&
	\wt{V} := \{(v,l)\in V\times \mathbb{P}(NS)\,|\, v\in  l \}.
	\nonumber
	\end{align}
\end{defn}

Note that the isomorphism class of the blow-up is independent of the choice of tubular neighbourhood. For later use, we will prove two lemmas on the behaviour of blow-up with respect to maps and to fibre products.

\begin{lem}
\label{lem:blow-up-maps}
	Let \(f\colon N\rar M\) be a map and let \(S_N\subset N\), \(S_M\subset M\) be closed submanifolds such that  \(S_N = f^{-1}(S_M)\). Then \(f\) lifts to a map \(\wt{f}\colon\wt{M}\rar\wt{N}\), where the blow-ups are with respect to \(S_M\) and \(S_N\) respectively. Moreover, if \(f\) is a submersion, then so is its lift.
\end{lem}
\begin{proof}
	Pick metrics on \(M\) and \(N\) such that \(f\) satisfies:
	\begin{align*}
	&
	\eta_N(X,Y) = \eta_M(df(X),df(Y)),
	&
	\forall X,Y\in (\ker df)^\bot.
	\end{align*}
	Note that if \(f\) is a submersion, this is the definition of \(f\) being a Riemannian submersion. Let \(j_M\colon NS_M\rar M\) and \(j_N\colon NS_N\rar N\) be tubular neighbourhoods obtained using the exponential map of the metrics. Then we get that:
	\begin{align*}
	j_N\circ df = f\circ j_M,
	\end{align*}
	since \(f_*\) sends geodesics to geodesics. Hence we can define \(\wt{df}\) as:
	\begin{align*}
	\wt{df}(v,[w]):=(df(v),[df(w)]).
	\end{align*}
	Gluing \(\wt{df}\) with \(f\) along the tubular neighbourhood gives \(\wt{f}\colon \wt{M}\rar\wt{N}\). The last claim follows immediately.
\end{proof}

\begin{lem}
	\label{lem:blow-up-fibred-product}
	Let \(f_A\colon A\rar C\) and \(f_B\colon B\rar C\) be transverse maps and let \(S_A\subset A\), \(S_B\subset B\) be smooth submanifolds such that \(S_C:=f_A(S_A) = f_B(S_B)\) is also a smooth manifold. Then:
	\begin{align*}
	\widetilde{A\times_C B} = \wt{A}\times_{\wt{C}}\wt{B}.
	\end{align*}
\end{lem}
\begin{proof}
	Since \(T(A\times_C B) \simeq TA\times_{TC} TB\), we get that \(N(S_A\times_{S_C} S_B) \simeq NS_A\times_{NS_C} NS_B\). Therefore, a choice of tubular neighbourhoods for \(S_A\) and \(S_B\) and \(S_C\) which are compatible, give a tubular neighbourhood \(j\) for \(S_A\times_{S_C} S_B\) by:
	\begin{align*}
	&
	j\colon NS_A\times_{NS_C} NS_B \rar A\times_C B,
	\\
	&
	(v_A,v_B)\mapsto (j_A(v_A), j_B(v_B)).
	\end{align*}
	Using this tubular neighbourhood for \(\widetilde{A\times_B C}\) gives the wanted isomorphism when we   extend it with the identity outside the tubular neighbourhoods, as all the maps in the following diagram commute:
	\begin{center}
		\begin{tikzcd}
			\widetilde{N(S_A\times_{S_C}S_B)} \arrow{d}{\pi} \arrow{r}{\simeq}
			&
			\wt{NS_A} \times_{\wt{NS_C}} \wt{NS_C} \arrow{d}{\pi_A\times_{\pi_C}\pi_B}
			\\
			N(S_A\times_{S_C} S_B) \arrow{d}{j} \arrow{r}{\simeq}
			&
			NS_A \times_{NS_C} NS_B \arrow{d}{j_A\times_{j_C} j_B  }
			\\
			A\times_C B \arrow{r}{id}
			&
			A\times_C B
		\end{tikzcd}
	\end{center}    
	This completes the proof.
\end{proof}

\subsection{Proper groupoids}
\label{subsec:proper groupoids}
Let \(\G\rrar M\) be a proper Lie groupoid and let \(S\subset M\) be a closed saturated submanifold. In this subsection, we will use the blow-up construction to construct a resolution of \(\G\). Let \(S\subset U \subset M\) and \(S\subset V\subset NS\) be a tubular neighbourhood of \(S\) such that  \(\G|_U \simeq (\G_S\times_S NS)|_V\), which exists by the linearization theorem,  \autoref{thm:lin}, and let \(\wt{M}\) be the blow-up of \(M\) along \(S\), using this tubular neighbourhood.

\begin{lem}
	The Lie groupoid \(\G\) acts on \(\wt{M}\).
\end{lem}
\begin{proof}
	Since \(\G\times_M (M\setminus S) \simeq \G|_{M\setminus S}\) we only have to define the action of \(\G\) on \(\wt{V}\) and check that it is compatible along the tubular neighbourhood \(\phi\). Note that:
	\begin{align*}
		\G|_U\times_U \wt{V} 
        &
        \simeq (\G_S\times_S NS)|_V \times_U \wt{V} 
        \simeq (\G_S\times_S V) \times_U \wt{V}
        \\
        &
        \simeq \G_S\times_S \wt{V}.
	\end{align*}
	Therefore, it is enough to show that \(\G_S\) acts on \(\wt{V}\), which is due to the action:
	\begin{align}
		g_s\cdot(v,[w]):=(g_s\cdot v, [g_s\cdot w]),
	\end{align}
	with \([w]\) the line through \(w\neq 0\). Note that if \(v\neq 0\), we have that \([w]=[v]\) and hence
	\begin{align*}
		\phi(g_s\cdot(v,[v])) 
        & 
		= \phi(g_s\cdot v, [g_s\cdot v]) 
        = \phi(g_s\cdot v)
        = g_s\cdot \phi(v),
	\end{align*}
	using the linearization. Therefore, the action is compatible with the action on \(U\setminus S\) of \(\G\).
\end{proof}

Another way to construct this action groupoid is to blow-up \(\G\) itself along the full sub-groupoid \(\G_S:=s^{-1}(S)=t^{-1}(S)\rrar S\).

\begin{lem}
	The action groupoid \(\G\times_M\wt{M}\) is isomorphic to the blow-up of \(\G\) along \(\G_S\).
\end{lem}
\begin{proof}
	Let \(\wt{\G}\) denote the blow-up of \(\G\) along \(\G_S=s^{-1}(S)\). Using  \autoref{lem:blow-up-maps}, we see that we can lift \(s\) to \(\wt{s}\colon\wt{\G}\rar\wt{M}\) and denote the blow-down map of \(\wt{\G}\) by \(\pi_\G\). Let \(\xi\colon\wt{\G}\rar \G\times_M \wt{M}\) be given by \(\pi_\G\times \wt{s}\). This map is an isomorphism. Indeed, one checks that it is an isomorphism away from \(\G_S\). Around \(\G_S\), we have an isomorphism of vector bundles:
	\begin{align*}
		N\G_S\simeq \G_S\times_S NS.
	\end{align*}
	Since the blow-up procedure is fiberwise, we also get:
	\begin{align*}
		\widetilde{N\G_S}\simeq \G_S\times_S \widetilde{NS}.
	\end{align*}
	Hence, when restricting both to a neighbourhood of \(\G_S\), we get that \(\wt{\G}\simeq \G\times_M \wt{M}\).
\end{proof}

\begin{defn}
	The \emph{blow-up of \(\G\) along \(S\)} is defined as the action groupoid
	\begin{align}
		\wt{\G}:=\G\times_M \wt{M} \rrar \wt{M}.
	\end{align}
\end{defn}

\begin{prop}
	\label{prop:blow-up-proper}
	Let \(\G\) be an proper Lie groupoid and let \(S\subset S^j\) with \(j=\max \{0\leq i \leq \dim(M) \,|\, s^i\neq\emptyset\} \) be a `most singular' stratum . Then the blow-up \(\wt{\G}\) of \(\G\) along \(S\) is a proper Lie groupoid. Moreover, the blow-down map \(\pi\colon\wt{\G}\rar\G\) is a surjective proper map, which is an isomorphism almost everywhere and the leaves \(\wt{L}\subset \pi^{-1}(S)\) satisfy:
	\begin{itemize}
		\item \(\pi(\wt{L})\) is a leaf in \(S\);
        \item \(\dim(\wt{L}) > \dim(\pi(\wt{L}))\).
	\end{itemize}
\end{prop}
\begin{proof}
	First note that since \(S\) is a most singular stratum, it is closed, and therefore the blow-up is well-defined. Let \(E:=\pi^{-1}(S)\subset\wt{M}\). By construction, away from \(E\), \(\wt{\G}\) is isomorphic to \(\G\) and hence away from \(E\), \(\pi\) is a proper surjective map and \(\wt{\G}\) is proper. Hence we are only considering what happens close to \(E\), where we have that \(\wt{\G}|_{\wt{U}}\simeq \G_S\times_S \wt{V}|_{\wt{V}}\). Since \(\pi\colon\wt{NS}\rar NS\) is a proper surjective map, so is its restriction \(\pi\colon\wt{V} = \pi^{-1}(V)\rar V\). As \(\wt{\G}\) is an action groupoid, it immediately follows that \(\wt{\G}\rar \G\) is surjective and proper.
	
	Consider the following diagram.
	\begin{center}
	\begin{tikzcd}
		\G_S\times_S \wt{V} 
        \arrow{r}{\pi}
        \arrow{d}{(\wt{s},\wt{t})}
        &
        \G_S \times_S V
        \arrow{d}{(s,t)}
        \\
        \wt{V}\times\wt{V}
        \arrow{r}{\pi\times\pi}
        &
        V \times V
	\end{tikzcd}
	\end{center}
	For a compact \(\wt{K}\subset \wt{V}\times\wt{V}\), we see that the set
	\begin{align*}
		(\wt{s},\wt{t})^{-1}(\wt{K}) \subset \pi^{-1} \left(   (s,t)^{-1}\left(  (\pi\times\pi)(\wt{K}) \right)\right)
	\end{align*}
	is a closed subset fo a  compact set and hence is compact. Therefore, \(\wt{\G}\) is a proper Lie groupoid.

	Now let \(\wt{L}\subset E\) be a leaf and let \((0,[w_i])\in\wt{L}\) for \(w_i\in NS\) and \(i=1,2\). Since they belong to the same leaf, there exists \(g_s\in \G_s\) such that \(g_s\cdot (0,[w_1])=(0,[w_2])\). Hence \(s(g_s) = \pi(w_1)\) and \(t(g_s) = \pi(w_1)\). That is, \(\pi(0,[w_1])\) and \(\pi(0,[w_2])\) belong to the leaf \(L\) and \(\pi(\wt{L})\subset L\subset S\). Now given any \(x\in L\), let \(y=\pi(0,[w])\in\pi(\widetilde{L})\) be arbitrary. Then there exists \(g_s\in \G_S\) such that \(g_s\colon y\rar x\). It follows that \(x = \pi(g_s\cdot(0,[w]))\in\pi(\wt{L})\) and hence \(\pi(\wt{L}) = L\).

	Finally, for \((0,[w])\in\wt{L}\) one can assume by rescaling that \(w\in V\subset NS\) and hence corresponds to an \(x\in U\subset M\). Then, locally, we see that \(\wt{L}\simeq L_x/\mathbb{Z}_2\) where the \(\mathbb{Z}_2\)-action comes from \([w]=[-w]\) and can be trivial. Hence the dimension of \(\wt{L}\) is equal to the dimension of \(L_x\). Now note that \(L_x\) does not lie in \(S\) and hence its dimension is bigger than that of \(\pi(\wt{L})\). 
\end{proof}

Note that \(\wt{\G}\) being proper is independent of \(S\) being a most singular stratum. In fact, any blow-up of a proper Lie groupoid along a closed saturated submanifold is once again proper. 
By blowing-up a finite amount of times we get:

\begin{thm}
\label{thm:action-resolution}
	Any proper Lie groupoid admits an action resolution.
\end{thm}
\begin{proof}
	Let \(j(\G):=\max\{0\leq i\leq \dim(M)\,|\, S^j\neq\emptyset\}\) and  \(m=\min\{0\leq i\leq \dim(M)\,|\, S^j\neq\emptyset\}\). If \(j - m = 0\), \(\G\) is regular and we are done. Assume by induction that all proper Lie  groupoids with \(j(\G)-m(\G) < k\) admit an action resolution. We will show that any proper Lie groupoid with \(j(\G)-m(\G) = k\) also admit an action resolution. The theorem then follows by induction.
	
	Hence let \(j(\G) -m (\G) = k\) for a groupoid \(\G\rrar M\). Since \(S^j\) is by construction saturated, we can define the blow-up \(\pi\colon\wt{\G}\rar\G\) of \(\G\) along \(S^j\). Property (2) of  \autoref{prop:properties-dim-strat} implies that we can pick the tubular neighbourhood around \(S^k\) as a disjoint union of tubular neighbourhoods of each connected component. Hence we can apply  \autoref{prop:blow-up-proper}, which implies that \(j(\wt{\G})-m(\wt{\G})<k\) and hence \(\wt{\G}\) admits a resolution \((\h\simeq \wt{\G}\times_{\wt{M}} N , p )\). Now, it is easily seen that the composition \(\pi\circ p\) is a surjective proper map, so that \(\h\) is a resolution for \(\G\) as well. Moreover:
	\begin{align*}
		\h 
        \simeq 
        \wt{\G}\times_{\wt{M}} N
        \simeq 
        (\G\times_M \wt{M})\times_{\wt{M}} N
        \simeq
        \G\times_M N
	\end{align*}
	And hence \(\h\) is an action resolution. 
\end{proof}

\begin{defn}
	An action resolution obtained by several blow-ups will be called a \emph{desingularization}. 
\end{defn}

Note that desingularizations are almost-everywhere isomorphic to the original groupoid.

\subsection{Functoriality}
\label{subsec:blow-up-functoriality}
In this section we will prove that generalized morphisms between proper Lie groupoids lift to generalized morphisms between their blow-ups, as long as we blow-up submanifolds which are somehow related. When we specialize to Morita equivalent groupoids, we can show more: their desingularizations are again Morita equivalent. Throughout, let \((P,\alpha',\alpha)\colon\G'\drar \G\) be a generalized map. Given a subset \(S'\subset M'\) or \(S\subset M\), we let \(P(S'):= \alpha( (\alpha' )^{-1}(S'))\subset M \) and \(P^{-1}(S):=\alpha'( \alpha^{-1}(S) )\subset M'\).

\begin{lem}
	If \(S'\subset M' \) and \(S\subset M\), then \(P(S')\) and \(P^{-1}(S)\) are saturated. Moreover, \(P(P^{-1}(S)) = {\rm Sat}(S)\), the saturation of \(S\).
\end{lem}
\begin{proof}
	Let \(x\in P(S')\) and \(g\colon y\rar x\in \G\). Then, by definition of \(S\), there exists \(p\in P\) such that \(x=\alpha(p)\) and \(\alpha'(p)\in S'\). Since \(\alpha'(p\cdot g) = \alpha'(p)\in S'\), we conclude \(y = \alpha(p\cdot g)\in P(S')\) and hence \(P(S')\) is saturated. Note that we do not use any assumptions of \(P\) being right-principal, so the same holds for \(P^{-1}(S)\). For the second claim, we have that \(S\subset P(P^{-1}(S))\) which is saturated. Hence \(\mbox{Sat}(S)\subset P(P^{-1}(S))\). Now for the reversed inclusion note that:
	\begin{align*}
		P(P^{-1}(S)) 
        = 
        \{
        \alpha(p)\,|\, \exists q\in P \mbox{ s.t. } \alpha'(p)=\alpha'(q) \mbox{ and } \alpha(q)\in S 
        \}.
	\end{align*}
	Since the action of \(\G\) on \(P\) has orbits equal to the fibers of \(\alpha'\), for \(\alpha(p)\in P(P^{-1}(S))\) we have that there exists \(g\in \G\) such that \(p=q\cdot g\). Hence \(\alpha(p) = \alpha(q\cdot g) \in \mbox{Sat}(S)\)  since \(\alpha(q)\in S\).
\end{proof}

In case the generalized morphism \(P\) is a Morita equivalence we also get that \(P^{-1}(P(S'))=\mbox{Sat}(S')\). We will call two saturated sets \(S'\) and \(S\) \emph{related} or, in case of a Morita equivalence, \emph{equivalent}, if \( S' = P^{-1}(S) \), which also immediately implies that \(S=P(S')\). The idea behind related sets is that if we want to compare blow-ups of groupoids, we have to blow-up along related sets.

\begin{prop}
	Let \((P,\alpha',\alpha)\colon\G'\drar \G\) be a generalized morphism between proper Lie groupoids, \(S'\subset M'\) and \(S\subset M\) be related closed saturated submanifolds and let \(\wt{\G}'\) and \(\wt{\G}\) be the blow-ups of \(\G'\) and \(\G\) along these submanifolds. Then \(P\) lifts to a generalized map \((\wt{P},\wt{\alpha}',\wt{\alpha})\colon\wt{\G}'\drar \wt{\G}\). 
\end{prop}
\begin{proof}
	Let \(S_P:=\alpha^{-1}(S)=(\alpha')^{-1}(S')\) and \(\wt{P}\)  be the blow-up of \(P\) along \(S_P\). Note that since \(\alpha'\) is a submersion, we get that \(S_P\) is indeed a submanifold. By  \autoref{lem:blow-up-maps}, the maps \(\alpha\) and \(\alpha'\) lift to maps \(\wt{\alpha}\) and \(\wt{\alpha}'\) such that \(\wt{\alpha}'\) is again a submersion. Similarly, we can lift the actions, by lifting the maps \(P\times_M\G\rar P \) and \(\G'\times_{M'} P \rar P\), using both \autoref{lem:blow-up-maps} and  \autoref{lem:blow-up-fibred-product}. Note that the actions commute and that the moments maps are invariant for the other action, since it holds on open dense subsets, where the blow-ups are isomorphic to the original manifolds.
	
	If \(\wt{g}=(g,\wt{x})\) and \(\wt{p}\cdot \wt{g} = \wt{p}\), we also get that \(\pi_P(\wt{p}) \cdot g = \pi_P(\wt{p})\) and since \(\G\) acts freely, \(g\) has to be \(1_{\pi(\wt{x})}\). Hence \(\wt{g} = (1_{\pi(\wt{x})},\wt{x}) = 1_{\wt{x}}\) and \(\wt{\G}\) acts freely. Using the properness of the \(\G\)-action, one can verify that the \(\wt{\G}\)-action is also proper. Hence we are left to show that \(\wt{M}'\simeq \wt{P}/\wt{Q}\) to conclude the proposition.
	
	Let \(\wt{\alpha}'(\wt{p})=\wt{\alpha}'(\wt{q})\in E'\subset \wt{M}'\). Then we can find \(\gamma_p,\gamma_q\colon I\rar P\) such that \(\alpha'\circ \gamma_p=\alpha'\circ\gamma_q\), \([\dot{\gamma_p}(0)]=\wt{p}\) and \([\dot{\gamma_q}(0)]=\wt{q}\). Hence there exists a unique \(\gamma_g\colon I\rar \G_S\) such that \(\gamma_p=\gamma_q\cdot \gamma_g\). Therefore \(\wt{q}\cdot \wt{g}=\wt{p}\) with \(\wt{g}:=[\dot{\gamma_g}(0)]\). Outside \(E'\), the claim is true as well. This proves the proposition.
\end{proof}

\begin{cor}
	\label{cor:morita-blow-up}
	Blow-ups of Morita equivalent groupoids along equivalent saturated submanifolds are Morita equivalent.
\end{cor}

In general it does not hold that any generalized morphism between groupoids lifts to their desingularizations. The main problem for this is that one cannot compare the stratifications. We can however when \(P\) is a Morita equivalence:

\begin{lem}
	\label{lem:strata-morita-equivalence}
	Let \((P,\alpha',\alpha)\colon\G'\rar \G\) be a Morita equivalence and \(p\in P\). Then:
	\begin{align*}
		\codim\left(L_{\alpha'(p)}'\right) = \codim\left(L_{\alpha(p)}\right)
	\end{align*}
	Hence a saturated \(S\subset M\) is a stratum if and only if \(P^{-1}(S)\) is.
\end{lem}
\begin{proof}
	Note that \(\alpha\) is a submersion which maps \(\G'\cdot p\cdot \G\) onto \(L_{\alpha(p)}\). Therefore \(\codim(L_{\alpha(p)}) = \codim(\G'\cdot p\cdot \G)\). The same holds for \(\alpha'\), which proves the lemma.
\end{proof}

In other words, there is a one-to-one correspondence between the strata of \(\G\) and \(\G'\). Note that this assignment of equivalent strata keeps the partial ordering of the stratification intact. That is, we can use  \autoref{cor:morita-blow-up} a finite amount of times to conclude:

\begin{thm}
	\label{thm:blow-up-morita}
	The desingularizations of Morita equivalent groupoids are Morita equivalent.
\end{thm}

\subsection{Properties}
\label{subsec:blow-up-properties}
Let \(\G\rrar M\) be a proper Lie groupoid and let \(\wt{\G}\rrar \wt{M}\) be its blow-up with blow-down map \(\pi\). In this section we will prove three lemma's on the behaviour of \(\wt{\G}\) with respect to \(\G\), which we will use in the next section on metrics. Besides \(\G\) and \(\wt{\G}\) being almost-everywhere diffeomorphic, they also have the same kind of \(s\)-fibres. Indeed:
	\begin{align*}
		\wt{s}^{-1}(\wt{x})
        =
        \G\times_M \{\wt{x}\} 
        \simeq 
        s^{-1}(\pi(x)).
	\end{align*}
Hence it makes sense to compare Haar systems. To this end we have the following lemma.

\begin{lem}
	\label{lem:Haar-system-blow-up}
	Let \(\mu\) be a proper Haar system on \(\G\). Then \(\wt{\mu}\) defined as	
	\begin{align*}
		\wt{\mu}^{\wt{x}}:=\mu^{\pi(\wt{x})},
	\end{align*}
	is a proper Haar system on \(\wt{\G}\).
\end{lem}
\begin{proof}
	We start by showing that \(\wt{\mu}\) is right-invariant. For this, let \((g,\wt{x})\in\wt{\G}\), let \(\wt{y}:=g\wt{x}\) and let \(f\) be a smooth function on \(\wt{s}^{-1}(x)\simeq s^{-1}(\xi(\wt{x}))\). Then, using the right-invariance of \(\mu\):
	\begin{align*}
		\int_{(h,\wt{y})\in\wt{s}^{-1}(\wt{y})} f(hg,\wt{x})\, d\wt{\mu}^{\wt{x}}(h,\wt{y}) 
        & = \int_{h\in s^{-1}(\xi(\wt{y}))} f(hg,\wt{x})\, d\mu^{\xi(\wt{y})}(h)    \\
        & = \int_{h'\in s^{-1}(\xi(x))} f(h',\wt{x}) \, d\mu^{\xi(\wt{x})}(h')  \\
        & = \int_{(h',\wt{x})\in \wt{s}^{-1}(\wt{x})} f(h',\wt{x}) \, d\wt{\mu}^{\wt{x}}(h',\wt{x}).
	\end{align*}
	
	Next we check the smoothness of the map \(\wt{x}\mapsto \int_{\wt{s}^{-1}(\wt{x})} f(\wt{g})\,d\mu^{\wt{x}}(\wt{g})\) for \(f\) smooth on \(\wt{\G}\). As \(\G\simeq \wt{\G}\) away from the exceptional divisor, we only have to prove it around \(S\) and \(E\). With a similar argument, one realizes that we only have to prove it on \(\xi^{-1}(x)\) with \(x\in S\) fixed.  In this case, however, the integration domain does not change and the smoothness of the map directly follows from the smoothness of \(f\) itself. Properness of \(\wt{\mu}\) follows now directly from properness of \(\mu\).
\end{proof}

Next, we consider actions of \(\G\) and lifts of actions. Suppose \(\G\) acts on a submersion \(f\colon N\rar M\) by \(\theta\). In this case \(\wt{\G}\) acts on the projection \(p_{\wt{M}}:\wt{N}:=N\times_M \wt{M}\rar \wt{M}\) through \((g,\wt{x})\cdot (y,\wt{x}) := (g\cdot y, g\cdot \wt{x})\) such that the other projection \(p_N\colon\wt{N}\rar N\) makes the following diagram commute.
	\begin{center}
	\begin{tikzcd}
		\wt{\G}\times_{\wt{M}}\wt{N} 
        \arrow{r}{\wt{\theta}}
        \arrow{d}{\pi\times p_N}
        &
        \wt{N}
        \arrow{d}{p_N}
        \\
        \G\times_M N
        \arrow{r}{\theta}
        &
        N
	\end{tikzcd}
	\end{center}
A similar result holds for their tangent lifts, which rely on a choice of connection. Similar to the previous lemma we have the following lemma.
\begin{lem}
	\label{lem:connection-blow-up}
	Let \(\sigma\) be a connection on \(\G\). Then \(\wt{\sigma}\) defined as
	\begin{align}
		\wt{\sigma}_{(g,\wt{x})}(  \wt{v}_{\wt{x}} ) 
        :=
        (\sigma_g\circ d\pi(\wt{v}_{\wt{x}}) ,\wt{v}_{\wt{x}})
        \in T_g\G \times_{T_{\pi(\wt{x})}M} T_{\wt{x}}\wt{M}
	\end{align}
	is a connection on \(\wt{\G}\).
\end{lem}
\begin{proof}
	As \(ds\circ \sigma={\rm id}\), we get that \(\wt{\sigma}\) lands in \(T\wt{G} = T\G\times_{TM}T\wt{M}\). Moreover, since \(d\wt{s}\) is just the projection onto the second factor in this fibred product, we also immediately get that \(d\wt{s}\circ\wt{\sigma} = {\rm id}\). Finally, note that \(\wt{u}(\wt{x}) = (u(\pi(\wt{x})),\wt{x})\) and hence:
	\begin{align*}
		\wt{\sigma}_{\wt{u}(\wt{x})}(\wt{v}_{\wt{x}}) 
        & 
        = 
        \left( \sigma_{u(\pi(x))} \circ d\pi(\wt{v}_{\wt{x}} ) ,  \wt{v}_{\wt{x}}\right) 
        \\
        & 
        = 
        \left( du\circ d\pi(\wt{v}
        _{\wt{x}}),\wt{v}_{\wt{x}}   \right) 
        = 
        d\wt{u}(\wt{v}_{\wt{x}}),
	\end{align*}
	so indeed $\tilde{\sigma}$ defines a connection.
\end{proof}
Using this, we get:
\begin{lem}
	Let \(\G\rrar M\) act on a submersion \(f\colon N\rar M\) by \(\theta\) and let \(\sigma\) be a connection on \(\G\). Then the tangent lift \(T\theta\) of \(\theta\) with respect to \(\sigma\) and the tangent lift \(T\wt{\theta}\) of \(\wt{\theta}\) with respect to \(\wt{\sigma}\) commute:
	\begin{center}
	\begin{tikzcd}
		\wt{\G}\times_{\wt{M}} T\wt{N}
        \arrow{r}{T\wt{\theta}}
        \arrow{d}{\pi\times dp_N}
        &
        T\wt{N}
        \arrow{d}{dp_N}
        \\
        \G\times_M TN
        \arrow{r}{T\theta}
        &
        TN
	\end{tikzcd}
	\end{center}
\end{lem}
\begin{proof}
	We compute:
	\begin{align*}
		dp_N\circ T\wt{\theta}( g,\wt{x},w_y,\wt{v}_{\wt{x}} )
        & 
       	= 
        dp_N\circ d\wt{\theta}
        \left( 
        \wt{\sigma}_{(g,\wt{x})}  \circ  dp_{\wt{M}} (w_y,\wt{v}_{\wt{x}}) ,(w_y,\wt{v}_{\wt{x}}) 
        \right)
        \\
        & 
        = 
        d\theta\circ d(\pi\times p_N)
        \left(  
        \wt{\sigma}_{(g,\wt{x})}( \wt{v}_{\wt{x}}),(w_y,\wt{v}_{\wt{x}})
        \right)
        \\
        & 
        =
        d\theta 
        \left(
        d\pi\circ \wt{\sigma}_{(g,\wt{x})}(\wt{v}_{\wt{x}}), dp_N(w_y,\wt{v}_{\wt{x}})
        \right)
        \\
        &
        =
        d\theta
        \left(
        \sigma_g\circ d\pi(\wt{v}_{\wt{x}}),w_y
        \right)
        \\
        &
        =
        d\theta
        \left(
        \sigma_g\circ df(w_y),w_y
        \right)
        \\
        &
        =
        T\theta(g,w_y)
        =
        T\theta \circ (\pi\times dp_N)(g,\wt{x},w_y,v_{\wt{x}}).
	\end{align*}
	This proves the commutativity.
\end{proof}

%
\section{Metrics and desingularizations}
\label{sec:metric}
In \cite{HoyoFernandes14} the authors prove the linearization results of proper Lie groupoids, by showing that any groupoid with a \(2\)-metric can be linearized and that any proper Lie groupoid admits a \(2\)-metric. Hence instead of blowing-up proper Lie groupoids, one can blow up Riemannian groupoids. Since we have shown that the blow-up of a proper Lie groupoid is again a proper Lie groupoid, results in \cite{HoyoFernandes14} show that there again exists a \(2\)-metric on the blow-up. We will show that we can use the metric on the original groupoid to construct an explicit metric on the blow-up. The advantage to this is that we obtain some extra properties on the behaviour of the blow-down map, in particular when restricted to the exceptional divisor.

\subsection{Groupoid metrics}
\label{subsec:groupoids-metrics}
Let us start by recalling some definitions and basic properties of metrics on groupoids, which can be found in \cite{HoyoFernandes14}.  The metrics we will use are the so-called simplicial metrics.

\begin{defn}
	A \emph{simplicial metric} on a groupoid \(\G\rrar M\) consists of a metric \(\eta^k\) on each component \(\G^{(k)}\) of the nerve of \(\G\) such that:
	\begin{itemize}
		\item Each face map \(\G^{(k)}\rar \G^{(k-1)}\) is a Riemannian submersion;
    	\item The group \(S_{k+1}\) acts by isometries on the \(\G^{(k)}\).
	\end{itemize}
	The pair \((\G,\eta)\) is called a \emph{Riemannian Lie groupoid}.
\end{defn}

An important result of \cite{HoyoFernandes14} is the existence of simplicial metrics on proper Lie groupoids:

\begin{thm}[c.f.\ \cite{HoyoFernandes14}]
	\label{thm:existence-simp.metric}
	Any proper Lie groupoid admits a simplicial metric.
\end{thm}

The proof of this theorem consists of two tricks, pushing down metrics and averaging metrics. These we will use later and hence we will now spend some time on discussing them. 
Firstly, suppose that \(f\colon N\rar M\) is a submersion and \(\eta_N\) a metric on \(N\). Then \(\eta_N\) is called \emph{\(f\)-transverse} if for all \(m\in M\) and all \(n_1,n_2\in f^{-1}(m)\) we have that the maps: 
	\begin{align*}
		(d_{n_1}f)^*\circ ((d_{n_2}f)^*)^{-1}\colon{\rm Ann} (T_{n_2} f^{-1}(m) ) 
		\rar 
		T_m^*M  
		\rar 
		{\rm Ann}(T_{n_1}  f^{-1}(m))
	\end{align*}
are isometries. In this case, there exists a unique metric \(\eta_M\) on \(M\) such that \(f\) is a Riemannian submersion. This metric, defined as
	\begin{align*}
		\eta_M( df(X),df(Y) ):=\eta_N( X,Y ),
	\end{align*}
for \(X,Y\in \ker(df)^\bot\), is called the \emph{push-forward metric}.

Secondly, recall that given a metric \(\eta\) on \(N\), its \emph{dual} \(\eta^*\) is pointwise defined as the dual of the map \(\eta_n\colon T_nN\rar T_n^*N\). In \cite{HoyoFernandes14}, del Hoyo and Fernandes use this dual to  \emph{average} \(\eta\), depending on a choice of Haar system and connection:

\begin{prop}[c.f.\ \cite{HoyoFernandes14}]
\label{prop:average-metric}
	Let \(\G\rrar M\) be a proper Lie groupoid acting free and properly on \((N,\eta)\rar M\) through \(\theta\) and let \(\pi\colon E\rar E/\G\) be the quotient map. Then a choice of Haar system \(\mu\) and connection \(\sigma\) defines \(\pi\)-transverse metric \({\rm Av}(\eta)\), whose dual is defined as:
	\begin{align}
		{\rm Av}(\eta)^*_n(\alpha,\beta) := \int_{g\in s^{-1}(x)} \eta_{n\cdot g}^*((T\theta_g)^* (\alpha), (T\theta_g)^*(\beta))\, \mu^x(g).
	\end{align}
	Moreover, if \(\eta\) was already \(\pi\)-transeverse, then the push-down metrics of \(\eta\) and \({\rm Av}(\eta)\) agree. Finally, if \(\G\) acts on \(N_i\rar M\) for \(i=1,2\)  and \(p\colon(N_1,\eta_1)\rar (N_2,\eta_2)\) is a Riemannian submersion which is equivariant with respect to these actions, then \(p\colon(N_1,{\rm Av}(\eta_1))\rar (N_2,{\rm Av}(\eta_2))\) is a Riemannian submersion as well. 
\end{prop}

In our construction of the blow-up of a groupoid, we use pull-backs. Metrics behave naturally under these pull-backs, as the following lemma shows.

\begin{lem}
\label{lem:pull-back-metric}
	Consider the pull-back diagram of manifolds:
	\begin{center}
	\begin{tikzcd}
		M\times_N M'
        \arrow{r}{p'}
        \arrow{d}{p}
        &
        (M',\eta')
        \arrow{d}{f'}
        \\
        (M,\eta)
        \arrow{r}{f}
        &
        (N,\eta_N)
	\end{tikzcd}
	\end{center}
	with \(f\) a Riemannian submersion. Such a pull-back admits a metric,
	\begin{align}
	\label{eq:defn:pull-back-metric}
    	\eta*\eta':= p^*(\eta) + p'^*(\eta') - (f\circ p)^*(\eta_N),
	\end{align}
	such that \(p'\) is a Riemannian submersion.	
\end{lem}
\begin{proof}
	For \(\eta*\eta'\) to be a metric, it is enough to show positive definiteness. Hence let \((v,v')\in TM\times TM'\) be such that \(df(v)=df'(v')\) and \(||(v,v')||=0\). Since \(f\) is a Riemannian submersion, we get that: 
	\begin{align*}
		0=\eta*\eta'((v,v'),(v,v')) \geq \eta'(v',v') \geq 0
	\end{align*}
	and hence \(v'=0\). It now readily follows that \(v=0\) so that \(\eta*\eta'\) is positive definite. Note that \(\ker(dp') = \ker(df)\times \{0\}\subset TM\times TM'\) and hence \(\ker(dp')^\bot = \ker(df)^\bot \times_{TN} TM'\). For \((v,v'),(w,w')\in \ker(dp')^\bot\) we get:
	\begin{align*}
		\eta*\eta'((v,v'),(w,w')) 
        &
        = 
        \eta(v,w) + \eta(v',w') - \eta_N(df(v),df(w)) 
        = 
        \eta(v',w') 
        \\
        &
        = 
        \eta(dp'(v,v'),dp'(w,w')).
	\end{align*}
	We conclude that \(p'\) is a Riemannian submersion.
\end{proof}

We will call the metric defined in  \autoref{eq:defn:pull-back-metric}, the \emph{pull-back metric} of \(\eta\) and \(\eta'\).
The  rest of this section is  structured  as follows: firstly, in  \autoref{subsec:metric-base}, we prove that \(\wt{M}\) admits a metric which satisfies some natural conditions. Then, mirroring the proof of  \autoref{thm:existence-simp.metric}, we use this metric to get a metric on the groupoid \(\tilde{\G}\) in  \autoref{subsec:metric-groupoid}. We finish in  \autoref{subsec:metric-morita} by showing that the constructed metric is Morita invariant.

\subsection{Metric on the base}
\label{subsec:metric-base}
We start with  focusing on the metric on the base \(M\). We will spend this section proving the following result:

\begin{prop}
	\label{prop:metric-base-blow-up}
	Let \(\G\rrar M\) be a proper groupoid endowed with a simplicial metric \(\eta\) and \(S\subset M\) be its most singular stratum. Then \(\widetilde{M}\), the blow-up of \(M\) along \(S\),  admits a metric \(\bar{\eta}\) such that:
	\begin{itemize}
		\item \(\pi\colon E\rar S\) is a Riemannian submersion when restricted to \(E=\pi^{-1}(S)\);
		\item \(\pi\colon\wt{M}\rar M\) is an isometry outside an open neighbourhood of \(E\).
	\end{itemize}
\end{prop}

Note that this result is similar to a Theorem of Alexandrino, \cite[Theorem 1.2]{Alexandrino10}, for a singular foliation. The proof we give is a simplification of that proof, where the simplification comes from the existence of the groupoid \(\G\). The proof will be used later on in \autoref{prop:blow-up-maps-Riemannian}.

\begin{proof}
	The proof starts with an alteration of \(\eta^0\) around \(E\). Let \(S\subset U\subset M\) be a neighbourhood such that \(\G\) is linearizable on \(U\): \(\G|_U\simeq \G_S\times_S NS\rrar NS\). By using a partition of unity, subordinated to \(U\) and \(S^c\), it is enough to show that on \(\wt{U}\) there exists a metric such that the first condition of the proposition holds.

	Firstly, note that for any \(v\in N_xS\setminus\{0\}\), we have an inclusion of \(\G_x\cdot v\subset O_v = s^{-1}(x)\cdot v\). This gives an inclusion on tangent spaces  \(T_v(\G_x\cdot v) \subset T_v O_v\subset T_v(NS)\). Using the metric on \(NS\), which is inherited by the one on \(M\), we can define:
	\begin{align}
    \label{eq:defn:bundle-H-metric-base}
		H_v:= (T_v(\G_x\cdot v) )^\bot \subset T_v O_v.
	\end{align}
	One can check that \(\dim(H_v)=\dim(L)\) for any leaf \(L\subset S\). Hence \(H\rar NS\setminus S\) is a vector bundle. Moreover, if we let \(p\colon NS\rar S\) be the projection, we see that \(H \cap \ker(dp)=0\). Let \(H'\rar NS\setminus S\) and \(B\rar NS\setminus S\) be defined as:
	\begin{align}
    \label{eq:defn:bundle-H2,B-metric-base}
		&
		H'
        := \left(H\oplus \ker(dp)\right)^\bot,
        &
        B
        := H\oplus H'.
	\end{align}
	Therefore we get that \(T(NS\setminus S) = B \oplus \ker(dp)\). Note that \(\ker(dp)\simeq B^\bot\) by the projection \({\rm pr}\colon TNS \rar B^\bot\). Hence \(\ker(dp)\) inherits a metric \({\rm pr}^*(\eta^0|_{B^\bot})\). Using this metric, \(\ker(dp)\) splits as \(K\oplus K^\bot\) with \(K\) a line bundle generated by \(\frac{d}{d\tau}|_{\tau =0} \tau v \).
	Hence we get \(T(NS\setminus S) = B\oplus K \oplus K^\bot\). With respect to this splitting, we can define:
	\begin{align*}
		\bar{\eta}^0_v 
   	 = 
    	\bpm    
    	p^*(\eta^0_S) 
    	& 
    	0 
    	& 
    	0
    	\\  
    	0 
    	&  
    	{\rm pr}^*(\eta^0|_{B^\bot}) 
    	& 
    	0 
    	\\ 
    	0 
    	& 
    	0 
    	& 
    	\frac{1}{|v|^2} {\rm pr}^*(\eta^0|_{B^\bot})
    	\epm .
	\end{align*}

	The next step is to extend \(\bar{\eta}^0\) over \(E= \pi^{-1}(S)\), to \(\wt{NS}\). Note that on \(\wt{NS}\setminus E\), we have that \(\pi\colon\wt{NS}\setminus E\rar U\setminus S\) is an isomorphism and hence we can use \(\bar{\eta}^0\). On \(E\), we see that \(T\wt{NS} = TE \oplus \wt{K}\), with \(\wt{K}\) a line bundle generated by paths of the form \(\gamma_v(\tau)=(\tau v,[v])\). To extend \(\bar{\eta}^0\), we ask \(TE\) and \(\wt{K}\) to be orthogonal, \(||\dot{\gamma_v}(0)|| = ||v||\) and on \(T_{(0,[v])}E\) we can view vectors as \((0, [V])\) for \(V\in T_v(NS)\) and we can define:
	\begin{align*}
		\bar{\eta}( (0,[V]), (0,[W]) )
        := \lim_{\tau\rar 0} \bar{\eta}^0_{\tau v} ( \tau V, \tau W ).
	\end{align*}
	Here \(\tau V \in T_{\tau v} (NS)\) is the derivative of the path \(\gamma_{\tau V}(\xi):= \tau \gamma_V(\xi)\).

	We are left to check that \(\pi\) is a Riemannian submersion. However, if \((0,[V])\in \ker (d\pi)\), then \(\tau V\in \ker(dp)\) for all \(\tau\). Hence on \(\ker(d\pi)^\bot\), we use \(p^*(\eta_S^0)\), from which it follows that \(\pi\) is a Riemannian submersion. 

\end{proof}

\subsection{Metric on the groupoid}
\label{subsec:metric-groupoid}
Using the metric \(\bar{\eta}\) we have just constructed on \(M\), we will construct one on \(\G^{(k)}\) for all \(k\). This construction mirrors the proof of Theorem \ref{thm:existence-simp.metric} in \cite{HoyoFernandes14}. 
First note that \(s\colon\G\rar M\) is a submersion and hence we have the submersion groupoid \(\G\times_M \G\rrar \G\). Denote the \((k-1)\)st nerve of this submersion groupoid by \(\G^{[k]}\). One easily realizes that there exists diffeomorphisms \(\psi_k\colon\G^{[k]}\rar \G^{(k)}\), for all \(k\), which do not form a groupoid map, defined by:
	\begin{align*}
		\psi_k(g_k,\hdots,g_1):=(g_kg_1^{-1},\hdots ,g_2 g_1^{-1}, g_1).
	\end{align*}
\begin{lem}
\label{lem:metric-submersion-groupoid-wrong}
	The metrics \(\psi_k^*(\eta^k)\) on \(\G^{[k]}\) form a simplicial metric on the submersion groupoid \(\G^{[2]}\rrar \G\).
\end{lem}
\begin{proof}
	One can easily check that the action of \(S_{k}\) on \(\G^{[k]}\) is included in the action of \(S_{k+1}\) on \(\G^{(k)}\) through \(\psi_k\). Similarly, each simplicial map of the groupoid \(\wt{\G}^{[k]}\rar \wt{\G}^{[k-1]}\) matches a simplicial map of \(\wt{\G}^{(k)}\rar \wt{\G}^{(k-1)}\) under \(\psi_k\). Note that the latter has one more simplicial map, but that does not matter. Hence it follows that the simplicial maps are Riemannian submersions, and the lemma follows. 
\end{proof}

Now using these, we can define a simplicial metric on \(\wt{\G}^{[k]}\):
\begin{lem}
	\label{lem:metric-submersion-groupoid}
	Let \(\bar{\eta}^0\) be the metric on \(\wt{M}\) of  \autoref{prop:metric-base-blow-up} and let \(\bar{\eta}^k\) on \(\wt{\G}^{[k]}\simeq \G^{[k]} \times_M \wt{M}\) be the pull-back metrics, as defined in \autoref{eq:defn:pull-back-metric}:
    \begin{align}
    \label{eq:metric-submersion-groupoid}
		& 
        \bar{\eta}^{k}
        := \psi_k^*(\eta^k) * \bar{\eta}^0.
    \end{align}
	Together, they form a simplicial metric on \(\wt{\G}^{[2]}\rrar\wt{\G}\). Moreover, the projections \(\pi^k\colon\wt{\G}^{[k]}\rar\G^{[k]}\) are Riemannian submersions when restricted to the exceptional divisor \(E\subset\wt{M}\).
\end{lem}
\begin{proof}
	Firstly, note that the action of \(S_k\) on \(\wt{\G}^{[k]}=\G^{[k]}\times_M \wt{M}\) restricts to the action of \(S_k\) on \(\G^{[k]}\). Hence by construction of the metric it follows directly that \(S_k\) acts by isometries using the same statement for \(\G^{[k]}\). Similarly, the face maps are just the face maps of \(\G^{[k]}\rar \G^{[k-1]}\) and we conclude that they are indeed Riemannian submersions. 
	
	Finally, when we restrict to \(E\) we get that \(\wt{\G}^{[k]}_E \simeq \G_S^{[k]}\times_S E\) and hence \(\pi^k_E\colon\wt{\G}^{[k]}_E\rar \G^{[k]}\) is the pull-back map of the Riemannian submersion \(\pi_E\colon E\rar S\), and therefore a Riemannian submersion itself, using  \autoref{lem:pull-back-metric}.
\end{proof}

Using maps \(\wt{\psi}_k\), similar to \(\psi_k\) but  for \(\wt{\G}\), we again get metrics on \(\wt{\G}^{(k)}\). The proof of \autoref{lem:metric-submersion-groupoid-wrong} shows that by using these \(\psi_k\) we lost some information: not all face maps were used for example. Hence the resulting metrics \(\wt{\psi}_k^*(\bar{\eta})\) do not form a simplicial metric. The rest of this subsection is devoted to altering these metrics such that they become simplicial.

\begin{prop}
	\label{prop:metric-groupoid-blow-up}
	Let \(\G\rrar M\) be a proper groupoid endowed with a simplicial metric \(\eta\) and let \(S\subset M\) be a saturated submanifold. Then \(\wt{\G}\), the blow-up of \(\G\) along \(S\), admits a metric \(\wt{\eta}\) such that:
	\begin{itemize}
		\item \(\pi\colon\wt{\G}|_E\rar \G_S\) is a Riemannian submersion when restricted to \(E=\pi^{-1}(S)\);
        \item \(\pi\colon\wt{\G}\rar \G\) is an isometry outside an open neighbourhood of \(E\).
	\end{itemize}
\end{prop}
\begin{proof}
	Let \(\bar{\eta}\) be the simplicial metric on \(\wt{\G}^{[k]}\) of  \autoref{lem:metric-submersion-groupoid} and let \(\psi_k^*(\eta^k)\) be the simplicial metric on \(\G^{[k]}\). Note that \(\G\) acts on the source map \(\G^{[k]}\rar M\) via 
	\begin{align*}
		(g_k,\hdots,g_1)\cdot g_0 = (g_k\cdot g_0,\hdots,g_1\cdot g_0)
	\end{align*}
	with quotient map \(\phi_k\colon\G^{[k]}\rar  \G^{(k-1)}\)
	\begin{align*}
		\phi_k(g_k,\hdots,g_1):= (g_k\cdot g_{k-1}^{-1},\hdots, g_2\cdot g_1^{-1})
	\end{align*}
	Similarly, \(\wt{\G}\) acts on \(\wt{\G}^{[k]}\). The following diagram shows our current situation:
	\begin{center}
	\begin{tikzcd}
		\hdots  
        \arrow[r, shift left] 
        \arrow{r} \arrow[r, shift right]
        & 
        (\wt{\G}^{[2]} ,\bar{\eta}^2)
        \arrow[r, shift left] \arrow[r, shift right] \arrow{d}{\wt{\phi}_2}
        &
        (\wt{\G} , \bar{\eta}^1)
        \arrow{r}{\pi^1} \arrow{d}{\wt{\phi}_1}
        &
        (\G, \psi_1^*(\eta^1)) 
        \arrow{d}{\phi_1}
        &
        (\G^{[2]} ,\psi_2^*(\eta^2)) 
        \arrow[l, shift left] \arrow[l, shift right] \arrow{d}{\phi_2}
        &
        \hdots \arrow[l, shift left] \arrow[l,shift right] \arrow{l}
        \\
        \hdots \arrow[r, shift left] 
        \arrow{r} \arrow[r, shift right]
        &
        \wt{\G} \arrow[r, shift left] 
        \arrow[r, shift right]
        &
        \wt{M} 
        \arrow{r}{\pi}
        & 
        (M,\eta^0) 
        &
        (\G,\eta^1)  
        \arrow[l, shift left] \arrow[l, shift right]
        &
        \hdots 
        \arrow[l, shift left] \arrow[l,shift right] \arrow{l}
	\end{tikzcd}
	\end{center}
	By  \autoref{lem:pull-back-metric} when we restrict the map \(\pi^k\) to \(E\subset M\) they become Riemannian submersions.  Now let \(\sigma\) be a connection  on \(\G\), \(\mu\) a Haar system on \(\G\) and let \(\wt{\sigma}\) and \(\wt{\mu}\) be the corresponding connection and Haar system on \(\wt{\G}\) as in \autoref{subsec:blow-up-properties}. We can now average the upper row, and push the metrics down to define:
	\begin{align}
		\label{eq:defn:blow-up-metric-groupoid}
		\wt{\eta}^k:=(\wt{\phi}_{k+1})_*(Av(\bar{\eta}^{k+1})).
	\end{align}
	Doing the same on the right hand side of the diagram leads to:
	\begin{center}
	\begin{tikzcd}[column sep=1.5em]
		\hdots  
        \arrow[r, shift left] \arrow{r} \arrow[r, shift right]
        & 
        (\wt{\G}^{[2]} ,{\rm Av}(\bar{\eta}^2))
        \arrow[r, shift left] \arrow[r, shift right] \arrow{d}{\wt{\phi}_2}
        &
        (\wt{\G} , {\rm Av}(\bar{\eta}^1))
        \arrow{r}{\pi^1} \arrow{d}{\wt{\phi}_1}
        &
        (\G, {\rm Av}(\psi_1^*(\eta^1)) )
        \arrow{d}{\phi_1}
        &
        (\G^{[2]} ,{\rm Av}(\psi_2^*(\eta^2)) )
        \arrow[l, shift left] \arrow[l, shift right] \arrow{d}{\phi_2}
        &
        \hdots \arrow[l, shift left] \arrow[l,shift right] \arrow{l}
        \\
        \hdots \arrow[r, shift left] 
        \arrow{r} \arrow[r, shift right]
        &
        (\wt{\G},\wt{\eta}^1) \arrow[r, shift left] 
        \arrow[r, shift right]
        &
        (\wt{M} ,\wt{\eta}^0)
        \arrow{r}{\pi}
        & 
        (M,\eta^0) 
        &
        (\G,\eta^1)  
        \arrow[l, shift left] \arrow[l, shift right]
        &
        \hdots 
        \arrow[l, shift left] \arrow[l,shift right] \arrow{l}
	\end{tikzcd}
	\end{center}
	Note that on the right hand side of the diagram, the metrics \(\psi_i^*(\eta^i)\) were already \(\phi_i\)-transverse with push-forward metrics \(\eta^i\) and hence do not change.
	
	We are left to show the two properties of \(\wt{\eta}\) of the proposition. Note that averaging and restricting to a saturated submanifold commute as we only use elements of \(\G\) and \(\wt{\G}\) which belong to the restriction. Therefore, away from \(E\),  \(\bar{\eta}^k  = \psi_k^*(\eta^k)\), \(\sigma=\wt{\sigma}\) and \(\mu=\wt{\mu}\), and hence the averaging and pushing down leads to \(\eta^i\). On \(E\), we know that \(\pi^k_E\colon(\wt{\G}_E^{[k]},\bar{\eta}^k_E)\rar(\G_S,\psi_k^*(\eta^k_S))\) is a Riemannian submersion. Let \(G=(g_k,\hdots,g_1)\in \G_S^{[k]}\) and \(G\cdot g = (g_k\cdot g,\hdots g_1\cdot g)\). We compute:
	\begin{align*}
		{\rm Av}
		&
		(\psi_k^*(\eta_S^k))_G^*
        (\alpha,\beta)
        =
        \int_{g\in s^{-1}(s(g_1))} 
        \!\!\!\!\!\!\!\!\!
        \psi_k^*(\eta_S^k)_{G\cdot g}^* \left( (T\theta_g)^* (\alpha), (T\theta_g)^*(\beta)  \right)\, \mu^{s(g_1)}(g)
        \\
        &
        =
        \int_{(g,\wt{x})\in \wt{s}^{-1}(s(g_1),\wt{x})}  
        \!\!\!\!\!\!\!\!\!\!\!\!\!\!\!\!\!\!\!\!\!\!\!\!\!\!\!
        \psi_k^*(\eta_S^k)_{G\cdot g}^* \left( (T\theta_g)^* (\alpha), (T\theta_g)^*(\beta)  \right) \, d\wt{\mu}^{\wt{s}(g_1,\wt{x})}(g,\wt{x})
        \\ 
        &
        =
        \int_{(g,\wt{x})\in \wt{s}^{-1}(s(g_1),\wt{x})}     
        \!\!\!\!\!\!\!\!\!\!\!\!\!\!\!\!\!\!\!\!\!\!\!\!\!\!\!
         (\bar{\eta}^k)^*_{(G,\wt{x})\cdot(g,\wt{x})}  \left( (d\pi^k)^* \circ (T\theta_g)^* (\alpha), (d\pi^k)^*\circ(T\theta_g)^* (\beta)  \right)      \, d\wt{\mu}^{\wt{s}(g_1,\wt{x})}(g,\wt{x})
        \\
        &
        =
        \int_{(g,\wt{x})\in \wt{s}^{-1}(s(g_1),\wt{x})}      
        \!\!\!\!\!\!\!\!\!\!\!\!\!\!\!\!\!\!\!\!\!\!\!\!\!\!\!
        (\bar{\eta}^k)^*_{(G,\wt{x})\cdot(g,\wt{x})}  \left( (T\wt{\theta}_{(g,\wt{x})})^*  \circ (d\pi^k)^* (\alpha), (T\wt{\theta}_{(g,\wt{x})})^*  \circ (d\pi^k)^* (\beta)  \right)      \, d\wt{\mu}^{\wt{s}(g_1,\wt{x})}(g,\wt{x})
        \\
        &
        =
        {\rm Av}(\bar{\eta}^k)^*_{(G,\wt{x})}\left((d\pi^k)^* (\alpha),(d\pi^k)^* (\beta)   \right).	
	\end{align*}
	Here the first and last equality are by definition of the average, the second by choice of \(\wt{\mu}\), the third by \(\pi_E\) being a Riemannian submersion before the averaging, and the fourth equation by the choice of \(\sigma\). This computation shows that \(\pi^k\) is still a Riemannian submersion after averaging. Therefore, three out of four maps of the following diagram are Riemannian submersions:
	\begin{center}
	\begin{tikzcd}
		(\wt{\G}^{[k]}_E,{\rm Av}(\bar{\eta}^k))
        \arrow{r}{\pi^k_E}
        \arrow{d}{\wt{\phi}_k}
        &
        (\G_S^{[k]}, {\rm Av}(\psi_k^*(\eta^k))  )
        \arrow{d}{\phi_k}
        \\
        (\wt{\G}^{(k-1)},\wt{\eta}^{k-1}) 
        \arrow{r}{\pi^{k-1}_E}
        &
        (\G^{k-1},\eta^k)
	\end{tikzcd}
	\end{center}
	Hence \(\pi^{k-1}_E\) is a Riemannian submersion.
\end{proof}

\begin{rem}
	Note that in the proof, we use the \(k\)th metric of \(\G\) to get the \((k-1)\)st metric on \(\wt{\G}\). This explains why we use simplicial metrics instead of \(2\)-metrics.
\end{rem}

The following theorem now follows.

\begin{thm}
\label{thm:metric-desingularization}
	Let \(\G\rrar M\) be a Riemannian proper groupoid. Then its desingularization \(\wt{\G}\) admits a simplicial metric \(\wt{\eta}\) such that:
	\begin{itemize}
		\item \(\pi\colon\wt{\G}|_E\rar \G_S\) is a Riemannian submersion when restricted to \(E=\pi^{-1}(S)\);
        \item \(\pi\colon\wt{\G}\rar \G\) is an isometry outside an open neighbourhood of \(E\).
	\end{itemize}
\end{thm}


\subsection{Morita invariance}
\label{subsec:metric-morita}
In the last part of this section, we combine all the previous results and prove that the desingularizations of two Morita equivalent proper Riemannian Lie groupoids are again Morita equivalent. This statement entails more than \autoref{thm:blow-up-morita}, as we also ask the metrics to be equivalent.

\begin{defn}
	A generalized morphism  \((P,\alpha',\alpha)\colon(\G',\eta')\drar(\G,\eta)\) with a metric \(\eta_P\) on \(P\) is called \emph{Riemannian} if \(\alpha'\)  is a Riemannian submersion. 
	Similarly, we define a \emph{Riemannian Morita equivalence} to be a Morita equivalence with a metric, such that both \(\alpha\) and \(\alpha'\) are Riemannian submersions.
\end{defn}

Note that since for any generalized morphism \(P\colon\G'\drar \G\) we have that \(M'\simeq P/\G\), and \(P\) being Riemannian implies that \(\eta_P\) is invariant under the \(\G\)-action.
Recall that generalized morphisms can also be defined by so-called \emph{fractions}, i.e., a groupoid maps \(\G'\lar \h\rar \G\), such that \(\h\rar \G'\) is a weak equivalence. In \cite{HoyoFernandes16} del Hoyo and Fernandes define a fraction to be \emph{Riemannian} if all the maps \(\h^{(k)}\rar \G'^{(k)}\) are Riemannian submersions for \(k\leq 2\). When working with simplicial metrics instead of \(2\)-metrics, it is natural to ask \(\h^{(k)}\rar \G'^{(k)} \) to be a Riemannian submersion for all \(k\).

The correspondence between generalized morphisms and fractions is given by \(\h = \G'\times_{M'}\times P \times_M \G\), where \(s(g',p,g) := p\cdot g\) and \(t(g',p,g)= g'\cdot p\), and the groupoid maps are given by the projections. We will refer to this Lie groupoid as the \emph{double action groupoid} corresponding to \(P\).  One can check that \(\h\rar \G'\) is indeed a weak equivalence and hence \(\h\), being Morita equivalent to a proper groupoid, is a proper Lie groupoid as well. Conversely, given a fraction \(\G'\lar \h\rar\G\), the corresponding bibundle is given by \(P:=\G'\times_{M'} N \times_M \G/\h\). In the next proposition we show that our notion of a Riemannian bibundle corresponds to a Riemannian fractions.

\begin{prop}
	Let \((P,\alpha',\alpha,\eta_P)\colon(\G',\eta')\drar(\G,\eta)\) be a Riemannian generalized morphism and let \(\h:=\G'\times_{M'}\times P\times_M \G\rrar P\) be the associated fraction. Then \(\eta_P\) induces a simplicial metric on \(\h\) for which \(\h\) is a Riemannian fraction. Conversely, if \(\G'\lar\h\rar\G\) is a Riemannian fraction, then \(\eta_\h\) induces a metric \(\eta_P\) on \(P:=\G'\times_{M'} N \times_M \G/\h\) such that  \((P,\eta_P)\) is a Riemannian bibundle.
\end{prop}
\begin{proof}
	For the first part of the proposition, we mimic the proof of  \autoref{prop:metric-groupoid-blow-up}.  \(\h^{[k]}\) can be written as \(\bar{\eta}_\h := \G'^{[k]}\times_{M'} P \times_M \G^{[k]}\) and hence caries the metric \(\psi'^*(\eta'^k) * \eta_P * \psi^*(\eta^k)\) with \(\psi\colon\G^{[k]}\rar \G^{(k)}\) the standard isomorphism. Using \(\psi'^*(\eta'^k)\) on \(\G'^{[k]}\), we get that \(\h^{[k]}\rar \G'^{[k]}\) is a Riemannian submersion if and only if \(P\times_M \G^{[k]}\rar M'\) is. This is the composition of the Riemannian submersion \(\alpha'\) with the projection map \(P\times_M \G^{[k]}\rar P\). The latter is a Riemannian submersion as well, which follows from \(t\colon \G\rar M\) being a Riemannian submersion. Hence \(\h^{[k]}\rar \G'^{[k]}\) is a Riemannian submersion.
	
	Now \(\h\) acts on \(\h^{[k]}\), but also on \(\G'^{[k]}\), by just using the \(\G'\) multiplication. Note that the metric for the latter action is invariant since it is already invariant for the \(\G'\)-action. By picking a Haar system and connection on \(\h\), we can average and push down both \(\bar{\eta}_\h\) and \(\psi'(\eta')\). Since the \(\h\) action on \(\G'^{[k]}\) preserves the metric, the resulting metric on \(\G'\) is again \(\eta'\). Moreover, the fact that the maps \(\h^{[k]}\rar \G'^{[k]}\) being Riemannian submersions implies that the maps \(\h^{(k)}\rar \G^{(k)}\) with the new metric on \(\h\) are Riemannian submersions as well. Hence the first part of the proposition follows.

	For the second part, using the Haar system and connection on \(\h\), we can average the pull-back metric \(\eta'*\eta_N* i^*(\eta)\) on \(\G'\times_{M'}N\times_M\G\) and then push it down to \(P\), to get a metric \(\eta_P\). The map \(P\rar M'\) is given by the quotient of the projection \(\G'\times_{M'} N \times_M \G\rar \G'\). Here \(\h\) acts on \(\G'\), by just considering the \(\G'\) component of \(\h\). Note that this projection is a double fibred pull-back of Riemannian submersions as in the following diagram:
	\begin{center}
    \begin{tikzcd}
    	\G'\times_{M'} N \times_M \G 
        \arrow{r}
        \arrow{dd}
        &
        N\times_M \G 
        \arrow{r}
        \arrow{d}
        &
        \G 
        \arrow{d}{t}
        \\
        &
        N 
        \arrow{r}
        \arrow{d}{\phi'}
        &
        M
        \\
        \G'
        \arrow{r}
        &
        M'
    \end{tikzcd}
	\end{center}
	Hence \(\G'\times_{M'}N \times_M\G\rar \G'\) is a Riemannian submersion for the metric \(\eta'*\eta_N*\eta\). Note that the metric on \(\G'\) is \(\h\) invariant, and hence averaging on \(\G'\) and pushing down to \(M'\) does not change the metric on \(M'\). Combining these two statements proves the second part of the proposition.
\end{proof}

Since extending the metric to the blow-up only works well when we blow up a closed stratum, we can not expect to lift any generalized morphism. However, when the generalized morphism is a Morita equivalence, \autoref{lem:strata-morita-equivalence} shows  we can compare strata and hence we can try to  lift the morphism. The following proposition shows that this is indeed possible.

\begin{prop}
	\label{prop:blow-up-maps-Riemannian}
	Let \((P,\alpha',\alpha,\eta_P)\colon(\G',\eta')\drar (\G,\eta)\) be a Morita equivalence between proper groupoids, \(S'\subset M'\) and \(S\subset M\) equivalent closed strata and let \(\wt{\G}'\) and \(\wt{\G}\) be the blow-ups of \(\G'\) and \(\G\) along these submanifolds. Then \(\wt{P}\), the generalized morphism between \(\wt{\G}'\) and \(\wt{\G}\), admits a  metric such that it is a Riemannian generalized morphism.
\end{prop}
\begin{proof}
	Let \(\h:= \G'\times_{M'}P\times_M \G\rrar P\) be the double action groupoid, corresponding to \(P\) and let \(S_P:=(\alpha')^{-1}(S')\), which is a closed stratum.  \autoref{prop:metric-base-blow-up} gives a metric \(\bar{\eta}_P\) on \(\wt{P}\) such that the blow-down map \(\pi_P\colon\wt{P}\rar P\) is an isometry outside a neighbourhood of \( E_P:=\pi^{-1}(S_P)\) and a Riemannian submersion when restricted to \(E_P\). We will first show that we can make the choices in the proof of  \autoref{prop:metric-base-blow-up} such that the maps \(\alpha\colon(P,\bar{\eta})\rar(M,\bar{\eta}^0)\) and \(\alpha'\colon(P,\bar{\eta})\rar(M',\bar{\eta'}^0)\) are Riemannian submersions. Afterwards, we alter the metric on \(\wt{P}\) such that the same holds for the metric \(\wt{\eta}\) on \(M\) and \(\wt{\eta}'\) on \(M'\).

	First of all, we can pick tubular neighbourhoods \(U\), \(U'\) and \(U_P\) of \(S\), \(S'\) and \(S_P\) on which the groupoids linearize, which satisfy \(U=\alpha(U_P)\) and \(U'=\alpha'(U_P)\). Similarly we can pick compatible partitions of unity. Since \(\alpha\) and \(\alpha'\) are Riemannian submersions and the metrics on the blow-ups agree with the metrics of the original manifold outside the tubular neighbourhoods, we see that \(\wt{\alpha}\) and \(\wt{\alpha}'\) are  Riemannian submersions outside \(\pi_P^{_1}(U_P)\) as well. Similarly, the metric is  \(\wt{\G}\)- and \(\wt{\G}'\)-invariant outside \(\pi_P^{-1}(U_P)\), sincriginal metric was \(\G\)- and \(\G'\)-invariant.  Hence we can restrict our attention to the tubular neighbourhoods. 

	As in the proof of  \autoref{prop:metric-base-blow-up}, we will start with proving the claims on \(\pi^{-1}_P( U_P\setminus S_P )\). Let \(B_P,H_P,H_P'\rar NS_P\) and \(B,H,H'\rar NS\)  be the vector bundles as in \autoref{eq:defn:bundle-H2,B-metric-base}. We will first show that the differential of  \(d\alpha\colon NS_P\rar NS\), which we will denote by \(T\alpha\), respects the splittings \(T(NS_P\setminus S_P) = B_P \oplus K_P \oplus K_P^\bot\) and \(T(NS\setminus S) = B\oplus K \oplus K^\bot\). First note that:
	\begin{align*}
		\dim(\ker(dp_P)) = \codim(S_P) = \codim(S) = \dim(\ker(dp)).
	\end{align*}
	Hence \(T\alpha\) is a fiber-wise linear isomorphism between \(K_P\oplus K_P^\bot\) and \(K\oplus K^\bot\). Since \(\ker T\alpha\subset B_P\), and the metrics on \(NS_P\) and \(NS\) is inherited by the ones on \(P\) and \(M\), we see that \(T\alpha\) is in fact a fiber-wise isometry with respect to the metrics \({\rm pr}_P^*(\eta_P|_{B_P^\bot})\) and \({\rm pr}^*(\eta^0|_{B^\bot})\). Moreover, since \(d\alpha\) is linear, it is clear that \(K_P\) is mapped onto \(K\). This proves that \(\bar{\eta}_P\) and \(\bar{\eta}^0\) agree on \(K\oplus K^\bot\), even with the correction term. Therefore, we are left to consider \(B_P\) and \(B_P'\). We can extend \(\alpha\) to a groupoid morphism \(\h\rar \G\) in the obvious way. Hence for any \(v_P\in N_pS_P\) and \(v:=d\alpha(v_P)\in N_xS\) we get 
	\begin{align*}
		d\alpha(h\cdot v_p) = \alpha(h)\cdot v,
	\end{align*}
	showing that \(d\alpha\) maps \(\h_p\cdot v_p\mapsto \G_x\cdot v\) and \(s^{-1}(p)\cdot v_p\mapsto \G_x\cdot v\). This also implies that \(H_P'\) is mapped to \( H'\).  Therefore \(T\alpha\colon B_P\rar B\) is Riemannian with respect to the original metric. Since \(S_P\rar S\) is Riemannian as well and twisting the metric on \(H\) to \(p^*(\eta_S^0)\) keeps \(T\alpha\) Riemannian. Hence, with the metric \(\bar{\eta}\) on \(NS\setminus S\) and \(\bar{\eta}_P\) on \(NS_P\setminus S_P\), the map \(\wt{\alpha}\) is still a Riemannian submersion. 

	Finally, we consider the extension over the exceptional divisors \(E_P\) and \(E\). Similar to before, the line bundle \(\wt{K}_P\) is mapped to the line bundle \(K\) by \(d\wt{\alpha}\) and the metrics agree here. On \(TE_P\), the limit definition of the metric on the orthogonal complement of \(K\) immediately implies that \(\wt{\alpha}\) is still Riemannian.  Note that the same argument holds for \(\wt{\alpha}'\), since we did not make any choices.

	In order to conclude the proposition, we let \(\wt{\eta}_P:=t_*(\mbox{Av}(\bar{\eta}^1_P))\), with \(\bar{\eta}^1_P = \bar{\eta}'^1*\bar{\eta}_P * i^*(\bar{\eta}^1)\) on \(\wt{\h} \simeq \wt{\G}'\times_{\wt{M'}} \wt{P}\times_{\wt{M}} \wt{\G}\), where \(\bar{\eta}^1\) is the metric of \autoref{lem:metric-submersion-groupoid}. The map \(\wt{\alpha'}\) is the quotient map of  
	\begin{align*}
  		(\wt{\G}' \times_{\wt{M}'} \wt{P} \times_{\wt{M}} \wt{\G},
        \bar{\eta'}^1*\bar{\eta}_P* i^*(\bar{\eta}^1))
        \rar (\G',\bar{\eta'}^1)
     \end{align*}   
	with respect to the right \(\wt{\h}\)-action. This map is a Riemannian submersion as it is the double pull-back of Riemannian submersions in the following diagram:
	\begin{center}
	\begin{tikzcd}
		\wt{\G}' \times_{\wt{M}'} \wt{P} \times_{\wt{M}} \wt{\G}
        \arrow{r}
        \arrow{dd}
        &
		\wt{P}\times_{\wt{M'}}\wt{\G}  
        \arrow{r}
        \arrow{d}
        &
        \wt{\G}
        \arrow{d}{\wt{t}}
        \\
        &
        \wt{P}
        \arrow{r}
        \arrow{d}{\wt{\alpha}'}
        &
		\wt{M}
        \\
        \wt{\G}'
        \arrow{r}
        &
        \wt{M}'
	\end{tikzcd}
	\end{center}
	Note that the metric on \(\wt{\G}'\) is \(\wt{\h}\)-invariant if and only if it is \(\wt{\G}'\), and hence averaging on \(\wt{\G}'\) and pushing down to \(\wt{M}'\) gives exactly \((\wt{\eta}')^0\). Using the standard arguments it now follows that \(\wt{\alpha'}\) is a Riemannian submersion. For \(\wt{\alpha}\) a similar argument holds. Note however that the \(i^*\) will be cancelled as we use the left action of \(\wt{\h}\) on \(\wt{\G}\) instead of the right \(\wt{\G}\) action.
\end{proof}

Using \autoref{lem:strata-morita-equivalence}, we  conclude:

\begin{thm}
\label{thm:riemannian-morita-desingularization}
	The Riemannian desingularizations of Morita equivalent Riemannian groupoids are again Morita equivalent.
\end{thm}

%

\section{Discussion}
\label{sec:discussion}
\subsection{Orbit space}
\label{subsec:orbitspace}
Let \(\G\rrar M\) be a proper Lie groupoid. In this last section we will discuss some consequences for the orbit space \(X:=M/\G\). 
The first result is that \(X\) carries the structure of a stratified space:
\begin{prop}
	The dimension stratification descends to a stratification of the orbit space \(X:=M/\G\).
\end{prop}
\begin{proof}
	First note that since all strata and all the subsets \(S^i\) are saturated, they indeed define subsets of \(X\). Moreover, it is easy to see that these subsets are smooth manifolds since they have charts given by the subspaces \((N_xL_x)^{\G_x^\circ}\subset N_xL_x\). For a leaf \(L = L_x\in X\), we have that \(X_L\), the stratum of \(X\) through \(L\), has dimension equal to \(\dim( (N_xL_x)^{\G_x^\circ})\). Using an argument similar to the one in the proof of  \autoref{prop:strat-dim}, we realize that any leaf \(L'=L_y\) close to \(L\), and not in the same stratum as \(L\),  has \((N_yL_y)^{\G_y^\circ}\subset  (N_xL_x)^{\G_x^\circ}\) and hence its stratum is of lower dimension. Now the Frontier condition for the stratification on  \(M\) implies the Frontier condition of this partitioning on \(X\).
\end{proof}

In \cite{PflaumPosthumaTang14}, the authors show that if \(\G\rrar M\) is a proper Lie groupoid and \(M\) carries a so-called transversely invariant metric, the orbit space \(X\) is a metric space, with the metric given by:
\begin{align}
\label{eq:metric-orbit-space}
d_X(L,L'):= \inf \{ d(x_1,L) + \hdots + d(x_n,L_{n-1})\,|\, n\in\mathbb{N},\, x_i\in L_i,\, \forall 1\leq i\leq n-1,\, x_n\in    L' \}.
\end{align}
It is not hard to show that the zeroth component of any simplicial metric \(\eta\) on \(\G\rrar M\) is in fact such a transversely invariant metric and hence we get:

\begin{prop}
	\label{prop:X-metric-space}
	Let \(\G\rrar M\) be a proper Lie groupoid, equiped with a simplicial metric. Then \(d_X\), as in  \autoref{eq:metric-orbit-space}, turns \(X:=M/\G\) into a metric space such that the quotient map \(M\rar X\) is a submetry.
\end{prop}

Even though \(\wt{X}\rar X\) turns out to be not a submetry, one can say something about the Gromov--Hausdorff distance between \(\wt{X}\) and \(X\) when \(M\) is compact. Note that in constructing \(\bar{\eta}^0\) on \(\wt{M}\) as in \autoref{prop:metric-base-blow-up}, one only adjusts the metric on a tubular neighbourhood. By picking smaller tubular neighbourhoods, Alexandrino shows in \cite{Alexandrino10} that for all \(\varepsilon >0\), and for all leaves \(\wt{L},\wt{L}'\subset \wt{M}\), one has:
	\begin{align*}
		|d(\wt{L},\wt{L}') - d(\pi(\wt{L}),\pi(\wt{L}')) | < \varepsilon,
	\end{align*}
with respect to the metric \(\bar{\eta}^0\) on \(\wt{M}\), which depends on \(\varepsilon\). One can show that this is still true after averaging  \(\bar{\eta}^0\) with respect to the \(\wt{\G}\)-action and hence we get:

\begin{prop}
\label{prop:gromov-hausdorff}
	Let \(\G\rrar M\) be a proper Lie groupoid with simplicial metric \(\eta\) and assume that \(M\) is compact. Let \(\wt{\G}\rrar \wt{M}\) be its blow-up. Then for all \(\varepsilon >0\), \(\wt{\G}\) admits a metric such that the Gromov--Hausdorff distance between \(X=M/\G\) and \(\wt{X}=\wt{M}/\wt{\G}\) is smaller than \(\varepsilon\).
\end{prop}

Applying this a finite amount of times, leads to:

\begin{cor}
\label{cor:gromov-hausdorff-desing}
	Let \(\G\rrar M\) be a proper Lie groupoid with simplicial metric \(\eta\) and assume that \(M\) is compact. Then for all \(\varepsilon>0\), its desingularization \(\wt{\G}\rrar\wt{M}\) admits a simplicial metric, depending on \(\varepsilon\), such that the Gromov--Hausdorff distance between \(X=M/\G\) and \(\wt{X}=\wt{M}/\wt{\G}\) is smaller than \(\varepsilon\).
\end{cor}

As the orbit space of the desingularization \(\wt(\G)\) is an orbifold, \autoref{cor:gromov-hausdorff-desing} shows that the the orbit space of a general proper Lie groupoid is a Gromov--Hausdorff limit of orbifolds. In the framework of Riemannian groupoids, it seems natural to ask whether every proper Riemannian groupoid is a ``Gromov--Hausdorff limit" of regular proper Riemannian groupoids. A careful definition of the ``Gromov--Hausdorrff limit" of Riemannian groupoids is in demand for such an interesting statement.

\subsection{Examples}
\label{subsec:examples}
In this section we discuss some examples. 
\begin{exa}
	Let \(G\) be a Lie group, acting properly on \(M\) and let \(\G:=G\times M \rrar M\) be its action groupoid. Then \(\G\) is a proper groupoid and already linear. For any \(S\subset M\), we have that \(\G_S = G\times S\) and hence \(\G_S\times_S \widetilde{NS} = G\times \widetilde{NS}\). That is, \(G\) acts on \(\wt{M}\) and \(\wt{\G}  = G\times \wt{M}\). This agrees with Duistermaat and Kolk's blow-up of proper Lie group actions in \cite{DuitstermaatKolk00}.
\end{exa}

\begin{exa}
Generalizing the previous example, we can consider a proper Lie groupoid action of a given (possibly non-proper) Lie groupoid $\G\rightrightarrows M$ on a submersion $f\colon N\to M$.  
By definition this means that the associated action groupoid $\G\ltimes N\rightrightarrows N$ is proper, and induces a stratifications on $N$ according to the theory described in \S \ref{sec:stratifications}. Once again, if we blow-up a closed saturated submanifold $S\subset N$, we find that $\G$ acts on the blow-up $\wt{M}$, so that the resolution is given by
\[
\widetilde{\G\ltimes N}\cong\G\ltimes\widetilde{N}.
\]
Our theory can therefore alternatively be viewed as generalizing the resolution of proper Lie group action to the case of proper Lie groupoid actions. Notice also that any proper Lie 
groupoid acts on its own base space, so in this sense this example is universal.

Recall that in the Baum--Connes conjecture for Lie groupoids \cite[\S II.10.$\alpha$]{Connes94}, a geometric cycle for a Lie groupoid $\G$ is given by proper, cocompact actions of $\G$
on $f\colon N\to M$, together with a $K$-theory class in $K_i(C^*(\G\ltimes T_\G N))$, where $T_\G N$ is the tangent bundle along the fibers of $f$. Since $\G\ltimes T_\G N\rightrightarrows T_\G N$ is also a proper Lie groupoid, such $K_0$-classes are essentially given by representations of this groupoid. We hope that our theory can help to study this geometric $K$-group
for Lie groupoids, and by this contribute to index theory.
\end{exa}

\begin{exa}
Another example of a proper Lie groupoid is constructed in \cite{GL13}. Here one starts with a regular Riemannian foliation $(M,\mathcal{F},g)$ and considers the action of the holonomy 
groupoid ${\rm Hol}(\mathcal{F})$ on the normal bundle $N_\mathcal{F}$ to the foliation. Since the foliation is Riemannian, this is given by an injective morphism of groupoids 
\[
{\rm Hol}(\mathcal{F})\to O(N_\mathcal{F}),
\]
where $O(N_\mathcal{F})\rightrightarrows M$ is the groupoid of isometric linear maps of the Riemannian vector bundle $N_\mathcal{F}\to M$.
Taking the closure in the groupoid on the right hand side produces a proper Lie groupoid over $M$. By the result in \cite{PflaumPosthumaTang14} this proper Lie groupoid induces a
singular Riemannian foliation with partition given by the closure of the leaves of the original foliation. This is the groupoid proof of Molino's theorem that the leaf closure of a Riemannian
foliation defines a singular Riemannian foliation.

In this context, the resolution of this groupoid (or rather the associated full resolution to a manifold with corners) was used in \cite{BKR10} to give a proof of the index theorem for 
the basic Dirac operator on a Riemannian foliation.

\end{exa}

\subsection{The Dixmier--Douady class}
In this final subsection we show how to bring the classification of regular Lie groupoids of \cite{Moerdijk02} into play. Let $\G\rightrightarrows M$ be a proper Lie groupoid. Its 
desingularization $\widetilde{\G}\rightrightarrows \widetilde{M}$ is a proper regular Lie groupoid, and, by the main result of \cite{Moerdijk02} defines an extension
\[
1\longrightarrow K\longrightarrow \widetilde{\G}\longrightarrow E\longrightarrow 1
\]
of a proper foliation groupoid $E$ by a bundle of compact Lie groups $K$ given by the connected components of the isotropy groups of $\tilde{\G}$ . 
If we assume $K$ to be a bundle of {\em abelian} Lie groups, $E$ acts on $\widetilde{\G}$ through the extension above.
The quotient $\widetilde{M}\slash E$ has the structure of an orbifold and the bundle $K$ descends to the quotient. In the case that $K$ has fibers $\mathbb{T}$ and is central in $\widetilde{\G}$, this is exactly a  
$\mathbb{T}$-gerbe with trivial band over the orbifold $\widetilde{M}\slash E$. 
In the general case of an bundle of abelian groups, Moerdijk in \cite{Moerdijk02} associates to such extensions a cohomology class
\[
\delta(\widetilde{\G})\in H^2(\widetilde{M}\slash E;\underline{K}),
\]
where $\underline{K}$ is the sheaf of smooth sections of the bundle $K$. This class is invariant under Morita equivalence and therefore we conclude:
\begin{thm}
Let $\G\rightrightarrows M$ be a proper Lie groupoid with abelian stabilizer groups on the principal stratum. Then $\G$ defines in a canonical way a cohomology class 
\[
\delta(\G)\in H^2(\widetilde{M}\slash E,\underline{K})
\]
\end{thm}
In case the band of the gerbe $K$ is trivial with fiber $\mathbb{T}$, one can use the exponential sequence to get the cohomology class in $H^3(M\slash E,\mathbb{Z})$, the so-called Dixmier--Douady class. It would be interesting to study proper Lie groupoids whose desingularization is a central extension of a proper \'etale groupoid, which is related to \cite{pmct2}. For nonabelian stabilizers, when the band of the gerbe is trivial, there is a natural class in $H^2(\widetilde{M}/E, Z(K))$, where $Z(K)$ is the center subgroup of $K$.
 It is an interesting open question what properties of the original groupoid $\G$ this cohomology class exactly measures.

%
\nocite{*} 
\bibliographystyle{hyperamsplain-nodash}
\bibliography{Resolutions}
\end{document}